\documentclass[11pt]{article}

\usepackage{amsmath,amssymb,amsthm,amscd,amsfonts,verbatim}
\usepackage{enumitem}
\usepackage[all,2cell]{xy}
\usepackage[margin=1in]{geometry}
\usepackage{thmtools,thm-restate}
\usepackage{calc}
\usepackage{hyperref}
\usepackage{booktabs}

\title{Homology of FI-modules}

\author{Thomas Church and Jordan S. Ellenberg
\thanks{The authors gratefully acknowledge support from the National Science Foundation. Church's work was supported in part by NSF grants DMS-1103807 and DMS-1350138, and by the Alfred P. Sloan Foundation. Ellenberg's work was supported in part by NSF grants DMS-1101267 and DMS-1402620, a Romnes Faculty Fellowship, and a John Simon Guggenheim Fellowship.}}

\newtheorem{innerthm}{Theorem}

\newenvironment{theorem-prime}[1]{\renewcommand\theinnerthm{\ref*{#1}$'$}\innerthm}{\endinnerthm}

\newtheorem{doubleinnerthm}{Theorem}

\newenvironment{theorem-doubleprime}[1]{\renewcommand\thedoubleinnerthm{\ref*{#1}$''$}\doubleinnerthm}{\enddoubleinnerthm}

\theoremstyle{plain}
\declaretheorem[parent=section]{theorem}
\declaretheorem[sibling=theorem,style=definition]{definition}
\declaretheorem[sibling=theorem,style=definition]{construction}
\declaretheorem[sibling=theorem]{proposition}
\declaretheorem[sibling=theorem]{lemma}

\declaretheorem[sibling=theorem]{corollary}
\declaretheorem[sibling=theorem,title=Proposition]{prop}
\declaretheorem[sibling=theorem,title=Remark,style=definition]{numberedremark}
\declaretheorem[numbered=no,title=Remark,style=definition]{unnumberedremark}
\declaretheorem[numbered=no,style=definition]{example}

\declaretheorem[name=Theorem]{maintheorem}

\declaretheorem[sibling=maintheorem,name=Corollary]{maincorollary}

\numberwithin{table}{section}

\newcommand{\nc}{\newcommand}
\nc{\dmo}{\DeclareMathOperator}

\nc{\Q}{\mathbb{Q}}
\nc{\Z}{\mathbb{Z}}
\nc{\N}{\mathbb{N}}
\newcommand{\C}{\mathcal{C}}
\nc{\CC}{\mathbb{C}}
\renewcommand{\AA}{\mathcal{A}}

\dmo\FI{FI}
\dmo\FIMod{FI-Mod}
\dmo\FB{FB}
\dmo\FBMod{FB-Mod}
\dmo\dMod{-Mod}
\nc{\ZMod}{\Z\dMod}
\nc{\CTMod}{\mathbb{C}[T]\dMod}
\nc{\CMod}{\mathbb{C}\dMod}
\nc{\op}{\text{op}}
\nc\FIop{\FI^{\op}}
\nc\FIopMod{\FI^{\op}\dMod}
\nc{\FIarrow}{\FI^{\rightsquigarrow}}

\dmo{\GL}{GL}
\dmo\Hom{Hom}
\DeclareMathOperator*{\colim}{colim}
\dmo\coker{coker}
\dmo\im{im}
\dmo\id{id}
\dmo\Sym{Sym}
\dmo\End{End}
\dmo\Ob{Ob}
\dmo\Inj{Inj}
\dmo\Bij{Bij}
\dmo\Tor{Tor}
\dmo\spn{span}
\dmo\Sets{Sets}
\dmo\Groups{Groups}
\dmo\Spec{Spec}

\nc\HH{\mathcal{H}}
\nc{\Ctilde}{\widetilde{C}}
\nc{\XX}{\widetilde{X}}
\nc\p{\mathfrak{p}}
\nc\q{\mathfrak{q}}
\nc{\chT}{\widetilde{C}}
\nc{\ch}{C}
\newcommand{\J}{\widetilde{J}}
\nc{\disjoint}{\sqcup}
\nc{\incThmD}{\iota} 
\nc{\iSV}{\iota} 
\nc{\iset}{i} 
\nc{\iso}{\simeq} 
\nc{\iB}{\varphi} 
\nc{\iC}{i} 
\nc{\comp}{\overline}

\nc{\ra}{\to}
\nc{\into}{\hookrightarrow}
\nc{\inj}{\hookrightarrow}
\nc{\onto}{\twoheadrightarrow}
\nc{\surj}{\twoheadrightarrow}

\nc{\cat}{\ref{catlabel}}
\nc{\cattag}{$\ast\ast$}

\renewcommand{\phi}{\varphi}
\renewcommand{\epsilon}{\varepsilon}
\nc{\bwedge}{{\textstyle\bigwedge}}
\nc{\coloneq}{\mathrel{\mathop:}\mkern-1.2mu=}

\nc{\abs}[1]{\left\lvert#1\right\rvert}
\newcommand{\set}[1]{\{#1\}}
\newcommand{\beq}{\begin{displaymath}}
\newcommand{\eeq}{\end{displaymath}}
\nc{\para}[1]{\medskip\noindent\textbf{#1.}}

\nc{\myarXiv}[1]{\href{http://arxiv.org/abs/#1}{arXiv:#1}.}
\nc{\mydoi}[1]{\href{http://dx.doi.org/#1}{doi:#1}.}
\nc{\myemail}[1]{\href{mailto:#1}{\nolinkurl{#1}}}

\newcommand{\proofofclaimend}{\hfill$\blacksquare$ \smallskip}

\begin{document}

\maketitle

\begin{abstract}
We prove an explicit and sharp upper bound for the Castelnuovo--Mumford regularity of an FI-module in terms of the degrees of its generators and relations.  We use this to refine a result of Putman on the stability of homology of congruence subgroups, extending his theorem to previously excluded small characteristics and to integral homology while maintaining explicit bounds for the stable range. This paper is freely available at \href{http://arxiv.org/abs/1506.01022}{http://arxiv.org/abs/1506.01022}.
\end{abstract}

\section{Introduction}

In recent years, there has been swift development in the study of various abelian categories related, in one way or another, to stable representation theory~\cite{CEF,CEFN,SamSnowdenGL,WiltshireGordon}.  The simplest of these is the category of {\em FI-modules} introduced in \cite{CEF}, which can be seen as a category of modules for a certain twisted commutative algebra.  A critical question about these categories is whether they are {\em Noetherian}; that is, whether a subobject of a finitely generated object is itself finitely generated.\footnote{In some contexts such abelian categories are called ``locally Noetherian'', the term ``Noetherian'' being reserved for categories where \emph{every} object is Noetherian. We  use ``Noetherian'' here in the broader sense, but we acknowledge that not every FI-module is finitely generated.}  

The category of FI-modules over $\Z$ is Noetherian \cite[Theorem A]{CEFN}, so any finitely generated FI-module $V$ can be resolved by finitely generated projectives.  One can ask for more\,---\,in the spirit of the notion of Castelnuovo--Mumford regularity from commutative algebra, one can ask for a  resolution of $V$ whose terms have explicitly bounded degree. Castelnuovo--Mumford regularity has proven to be a very useful invariant in commutative algebra, and we expect the same to be the case in this twisted commutative setting.  In the present paper, we prove a strong bound for the Castelnuovo--Mumford regularity of FI-modules, and explain how this regularity theorem allows us to refine a result of Putman~\cite{Putman} on the homology of congruence subgroups.  Although much of the paper is homological-algebraic in nature, the heart of the main results is \autoref{main:saturation}; this is a basic structure theorem for FI-modules, whose proof at the core boils down to a combinatorial argument on injections from $[d]$ to $[n]$ involving certain sets of integers enumerated by the Catalan numbers.

The theorems we obtain with these combinatorial methods naturally hold for FI-modules with coefficients in $\Z$. This is in contrast with earlier representation-theoretic approaches, which tend to apply only to FI-modules with coefficients in a field, usually required to have characteristic $0$.  On the other hand, the approach via representation theory provides a very beautiful theory unifying the study of many different categories  (see e.g.\ \cite{SS:stability}), while the arguments of the present paper are quite specific to FI-modules.  It would be very interesting to understand the extent to which the combinatorics in \S\ref{sec:combinatorics} can be generalized beyond FI-modules to the family of stable representation categories considered by Sam and Snowden.

\para{Notation} FI is the category of finite sets and injections; an \emph{FI-module} $W$ is a functor $W\colon \FI\to \ZMod$. Given a finite set $T$, we write $W_T$ for $W(T)$. For every $n\in \N=\{0,1,2,\ldots\}$ we set $[n]\coloneq \set{1,\ldots,n}$, and we write $W_n$ for $W_{[n]}=W([n])$. 

When $W$ is an FI-module 
we write $\deg W$ for the largest $k\in \N$ such that $W_k\neq 0$. To include edge cases such as $W=0$, we formally   define $\deg W\in \{-\infty\}\cup \N\cup \{\infty\}$ by
\[\deg W\coloneq \inf\big\{k\in \{-\infty\}\cup \N\cup \{\infty\}\,\big\vert\,W_n=0 \text{ for all }n>k\in \N\big\}.\]

\para{FI-homology} The functor $H_0\colon \FIMod\to \FIMod$ captures the notion of ``minimal generators'' for an FI-module. Given an FI-module $W$, the FI-module $H_0(W)$ is the quotient of $W$ defined by
\[H_0(W)_T\coloneq W_T\ /\ \spn\big(\im f_*\colon W_S\to W_T\,|\,f\colon S\into T, \abs{S}<\abs{T}\big).\] This is the largest FI-module quotient of $W$ with the property that all maps $f_*\colon H_0(W)_S\to H_0(W)_T$ with  $\abs{S}<\abs{T}$ are zero. An FI-module $W$ is \emph{generated in degree $\leq m$} if $\deg H_0(W)\leq m$.

The functor $H_0$ is right-exact, and we define $H_p\colon \FIMod\to \FIMod$ to be its $p$-th left-derived functor. One can think of $H_p(W)$ as giving minimal generators for the ``$p$-th syzygy" of the FI-module $W$. 
Our first main theorem bounds 
$H_p(W)$ in terms of 
$H_0(W)$ and $H_1(W)$. 
\begin{maintheorem}
\label{main:regularity}
Let $W$ be an FI-module with $\deg H_0(W) \leq k$ and $\deg H_1(W) \leq d$. Then $W$ has 
regularity $\leq k+d-1$: that is, for all $p>0$ we have
\[\deg H_p(W)\leq p+k+d-1.\]
\end{maintheorem}
It is natural to shift our indexing by writing $d_p(W)\coloneq \deg H_p(W)-p$; with this indexing, \autoref{main:regularity} states simply that $d_p(W)\leq d_0(W)+d_1(W)$.

We will define in \S\ref{section:free-and-generation} the notion of a \emph{free FI-module}, and we will see that $H_p(W)$ can  be computed explicitly from a free resolution of $W$. For now, we record one corollary:
\begin{maincorollary}
\label{maincor:kernel}
Let $M$ be a free FI-module generated in degree $\leq k$, and let $V$ be an arbitrary FI-module generated in degree $\leq d$. For any homomorphism $V\to M$, the kernel $\ker(V\to M)$ is generated in degree $\leq k+d+1$.
\end{maincorollary}

\para{Uniform description of an FI-module} Our next main result is the following theorem, which gives a uniform description of an FI-module in terms of a explicit finite amount of data.
\begin{maintheorem}
\label{main:colimit}
Let $W$ be an arbitrary FI-module, and let $N\coloneq\max\big(\!\deg H_0(W),\,\deg H_1(W)\big)$. Then for any finite set $T$,
\begin{equation}
\label{eq:colimit}
W_T=\colim_{\substack{S\subset T\\\abs{S}\leq N}} W_S. 
\end{equation}
Moreover, $N$ is the smallest integer such that \eqref{eq:colimit} holds for all finite sets.
\end{maintheorem}
We deduce \autoref{main:colimit} from \cite[Corollary 2.24]{CEFN}, by showing that the complex $\widetilde{S}_{-\ast}W$ we introduced there computes the FI-homology $H_*(W)$. An alternate proof of \autoref{main:colimit} has recently been given by Gan--Li \cite{GanLiCentral}; in contrast with our approach via FI-homology, they prove directly that an FI-module that is presented in finite degree admits a description as in~\eqref{eq:colimit}.

\para{Homology of congruence subgroups}
As an application of these theorems we have the following result on the homology of congruence subgroups, which strengthens a recent theorem of Putman~\cite{Putman}. For $L\neq 0\in \Z$, let $\Gamma_n(L)$ be the level-$L$ principal congruence subgroup \[\Gamma_n(L)\coloneq \ker\big(\GL_n(\Z)\to \GL_n(\Z/L\Z)\big).\] For $S\subset [n]$, let $\Gamma_S(L)\subset \Gamma_n(L)$ be the subgroup
\[\Gamma_S(L)\coloneq \{M\in \Gamma_n(L)\,\vert\,M_{ij}=\delta_{ij}\text{ if }i\notin S\text { or }j\notin S\}.\]
Notice that if $\abs{S}=m$, the subgroup $\Gamma_S(L)$ is isomorphic to $\Gamma_m(L)$.
\begin{maintheorem}
\label{main:congruence} 
For all $L\neq 0\in \Z$, all $n\geq 0$, and all $k\geq 0$, \[H_k(\Gamma_n(L);\Z)=\colim_{\substack{S\subset [n]\\\abs{S}<11\cdot 2^{k-2}}} H_k(\Gamma_S(L);\Z).\]
\end{maintheorem}
In fact, we prove a version of \autoref{main:congruence} for any ring satisfying one of Bass's stable range conditions; see \autoref{thm:congruencestronger} in \S\ref{sec:congruence}. 
This theorem has already been used by Calegari--Emerton \cite[\S5]{CalegariEmerton} to prove stability for the completed homology of arithmetic groups.

The conclusion of  \autoref{main:congruence} is based on the main result of Putman in \cite{Putman} on ``central stability'' for $H_k(\Gamma_n(M);\Z)$, but its formulation here is a combination of \cite[Theorem~B]{Putman} and our earlier theorem with Farb and Nagpal \cite[Theorem~1.6]{CEFN}. Our main improvement over Putman is that \autoref{main:congruence} applies to homology with integral coefficients (or any other coefficients), while \cite{Putman} only applied to coefficients in a field of characteristic $\geq 2^{k-2}\cdot 18-3$. This limitation was removed in \cite{CEFN}, but at the cost of losing any hope of an explicit stable range. The methods of the present paper maintain the applicability to arbitrary coefficients while recovering Putman's stable range.

\para{Ingredients of \autoref{main:congruence}} In light of \autoref{main:colimit}, to obtain the conclusion of \autoref{main:congruence} we must bound the degree of $H_0$ and $H_1$ for the FI-module $\HH_k$ satisfying $(\HH_k)_n=H_k(\Gamma_n(L);\Z)$. The key technical ingredients are \autoref{main:regularity} and a theorem of Charney on a congruence version of the complex of partial bases. We obtain in \autoref{pr:Gammacomplex} a spectral sequence with $E^2_{pq}=H_p(\HH_q)$. Charney's theorem tells us that this spectral sequence converges to zero in an appropriate sense, and \autoref{main:regularity} then lets us work backward to conclude that $E^2_{pq}$ vanishes outside the corresponding range, giving the desired bound on the degree of $H_0(\HH_q)$ and $H_1(\HH_q)$.

\begin{unnumberedremark}
The argument of \autoref{main:congruence} bears an interesting resemblance to that of the second author with Venkatesh and Westerland in \cite{EVW}.    In that paper, one proves a stability theorem for the cohomology of Hurwitz spaces, which cohomology carries the structure of module for a certain graded $\Q$-algebra $R$.  As in the present paper (indeed most stable cohomology theorems), the topological side of the argument requires proving that a certain complex, carrying an action of the group whose cohomology we wish to control, is approximately contractible.  The algebraic piece of \cite{EVW} involves showing that $\deg \Tor^R_i(M,\Q)$ can be bounded in terms of $\deg \Tor^R_0(M,\Q)$ and $\deg \Tor^R_1(M,\Q)$ \cite[Prop 4.10]{EVW}.  Exactly as in the proof of \autoref{main:congruence}, it is these bounds that allow us to carry out an induction in the spectral sequence arising from the quotient of the highly connected complex by the group of interest.
\end{unnumberedremark}

\para{Combinatorial structure of FI-modules}
Our last theorem is a basic structural property of FI-modules; this structural theorem provides the technical foundation for our other results, and is also of independent interest in its own right.

An FI-module $M$ is \emph{torsion-free} if for every injection $f\colon S\into T$ between finite sets, the map $f_*\colon M_S\to M_T$ is injective. 
In this case, for any subset $S\subset T$, we may regard $M_S$ as a submodule of $M_T$, by identifying it with its image under the canonical inclusion.

\begin{maintheorem}
\label{main:saturation} Let $M$ be a torsion-free FI-module generated in degree $\leq k$, and let $V\subset M$ be a sub-FI-module generated in degree $\leq d$. Then for all $n > \min(k,d)+ d$ and any $a \leq n$, 
\[V_n\cap \big(M_{[n]-\{1\}}+\cdots+M_{[n]-\{a\}}\big)=V_{[n]-\{1\}}+\cdots+V_{[n]-\{a\}}.\]
\end{maintheorem}

\autoref{main:saturation} holds for any $d\geq 0$ and $k\geq 0$. However in the cases of primary interest, we will have $k < d$, so in practice the threshold for \autoref{main:saturation} will be $n > k+d$.  We note also that 
\autoref{main:saturation} is trivially true for $a > d$: the inclusion $V_{[n]-\{1\}}+\cdots+V_{[n]-\{a\}} \subset V_n\cap (M_{[n]-\{1\}}+\cdots+M_{[n]-\{a\}}) $ always holds, and it is easy to show that 
$V_{[n]-\{1\}}+\cdots+V_{[n]-\{a\}} = V_n$ when $a>d$.

\para{Stating \autoref{main:saturation} without $M$} Although \autoref{main:saturation} seems to be a theorem about the relation between the FI-module $M$ and its submodule $V$, actually the key object is $V$; the role of $M$ is somewhat auxiliary. In fact, in \S\ref{section:ideal-Im} we will give a more general formulation in \autoref{thm:saturationalternate} that makes no reference to $M$; in place of the intersection $V_n\cap (M_{[n]-\{1\}}+\cdots+M_{[n]-\{a\}})$, we use the subspace of $V_n$ annihilated by the operator $\prod_{i=1}^a \big(\id -\, (i\ \,n+i)\big)\in \Z[S_{n+a}]$  (see \S\ref{section:ideal-Im} for more details). \autoref{thm:saturationalternate} is stronger than \autoref{main:saturation} and has content even in the case corresponding to $V=M$, when \autoref{main:saturation} says nothing. Aside from their application to FI-homology in this paper, these results are fundamental properties of the structure of FI-modules, and should be of interest on their own.

\para{\autoref{main:saturation} and homology} We will show that if $M$ is a free FI-module, \autoref{main:saturation} has a natural homological interpretation as a bound on the degree of vanishing of a certain derived functor applied to $M/V$; see \autoref{rem:saturationreinterpreted} and \autoref{cor:freeacyclic} for details. It is this interpretation that allows us to connect \autoref{main:saturation} with the bounds on regularity in \autoref{main:regularity}.

\para{Bounds on torsion} The conclusion of \autoref{main:saturation} can be phrased as a statement about the quotient FI-module $M/V$, and in the case $a=1$ this conclusion becomes particularly simple: it states that the map $(M/V)_{[n]-\{1\}}\to (M/V)_{[n]}$ is injective when $n>\min(k,d)+d$. This yields the following corollary of \autoref{main:saturation}. In general, an FI-module $W$ is \emph{torsion-free in degrees $\geq m$} if the maps $f_*\colon M_S\to M_T$ are injective whenever $\abs{S}\geq m$. 
\begin{maincorollary}
\label{main:torsionbound}
If $M$ is torsion-free and generated in degree $\leq k$, and its sub-FI-module $V$ is generated in degree $\leq d$, then the quotient $M/V$ is torsion-free in degrees $\geq \min(k,d)+d$.
\end{maincorollary}
\para{Alternate proofs of \autoref{main:regularity}}
In the time since this paper was first posted, alternate proofs of \autoref{main:regularity} have been given by Li \cite{Li} (based partly on Li-Yu \cite{LiYu}) and Gan \cite{Gan}. The structure of those proofs is different from ours. In this paper, we prove \autoref{main:saturation} in a self-contained way, and then deduce \autoref{main:regularity} as a direct consequence. By contrast, both Li and Gan use \autoref{main:torsionbound} as a stepping stone (replacing the need for the full strength of \autoref{main:saturation}); they prove both \autoref{main:regularity} and \autoref{main:torsionbound} together, using an induction on $k$ that bounces back and forth between those two results.

\para{Sharpness of \autoref{main:saturation}} 
Before moving on, we give a simple example showing that the bound of $\min(k,d)+d$ in \autoref{main:saturation} and \autoref{main:torsionbound} is sharp.

\begin{example}
Fix any $k\geq 0$ and any $d>k$. Let $M$ be the FI-module over $\Q$ such that $M_T$ is freely spanned by the $k$-element subsets of $T$. The FI-module $M$ is torsion-free and generated in degree~$k$ by $M_k\iso \Q$.

For any $d$-element set $U$, consider the element $v_U\coloneq \sum_{\substack{S\subset U\\\abs{S}=k}}e_S$. Let $V\subset M$ be the sub-FI-module such that $V_T$ is spanned by the elements $v_U\in M_T$ for all $d$-element subsets $U\subset T$. 
The FI-module $V$ is generated by $v_{[d]}\in V_d\iso \Q$, so 
\autoref{main:torsionbound} asserts that the quotient $W\coloneq M/V$ should be torsion-free in degrees $\geq k+d$.

In fact, we have $W_n\neq 0$ for $n<k+d$ and $W_n=0$ for $n\geq k+d$, which we can verify as follows.
By definition $V_n$ is spanned by the $\binom{n}{d}$ elements $v_U$ 
as $U$ ranges over the $d$-element subsets $U\subset [n]$, so the dimension of $V_n$ is at most $\binom{n}{d}$. When $n < k+d$ we have $\dim V_n\leq \binom{n}{d} < \binom{n}{k} = \dim M_n$  so $V_n \neq M_n$, verifying the first claim. On the other hand, with a bit of work one can check directly that $V_{k+d}=M_{k+d}$, which then implies $V_n = M_n$ for all $n \geq k+d$, verifying the second claim.

Since $W_n=0$ for $n\geq k+d$ we see that $W$ is torsion-free in degrees $\geq k+d$, as guaranteed by \autoref{main:torsionbound}; however, the fact that $W_{k+d-1}\neq 0$ shows that this bound cannot be improved.
\end{example}

\para{Castelnuovo--Mumford regularity}
What we prove in \autoref{main:regularity} is that
\begin{displaymath}
\deg H_p(V) \leq c_V + p
\end{displaymath}
for some constant $c_V$ depending on $V$.  By analogy with commutative algebra, this statement could be thought of as saying that the {\em Castelnuovo--Mumford regularity} of $V$ is at most $c_V$.  For FI-modules over fields of characteristic $0$, that all finitely generated FI-modules have finite Castelnuovo--Mumford regularity in this sense is a recent result of Snowden and Sam~\cite[Cor 6.3.4]{SamSnowdenGL}.

We emphasize that \autoref{main:regularity} gives an explicit description of the regularity of $V$ which depends only on the \emph{degrees} of generators and relations for $V$.  This is much stronger than the bounds for the regularity of finitely generated modules $M$ over polynomial rings $\CC[x_1, \ldots, x_r]$, which depend on the \emph{number} of generators of $M$.  We take this strong bound on regularity as support for the point of view that the category of FI-modules is in some sense akin to the category of graded modules for a univariate polynomial ring $\CC[T]$. (See \autoref{table:analogy} for more details of this analogy.) Of course, in the latter context the fact that the regularity is bounded by the degree of generators and relations is a triviality because $H_p(V) = 0$ for all $p > 1$; by contrast, the category of FI-modules has infinite global dimension.

Despite these analogies, we would like to emphasize one surprising feature of the bound
\[\deg H_p(V)-p\leq \deg H_0(V)+\deg H_1(V)-1\]
we obtain for FI-modules: one cannot expect a bound of this form to hold for graded modules over a general graded ring, for the simple reason that the bound is not invariant under shifts in grading.\footnote{Indeed, since $\deg H_i(V[k])=\deg H_i(V)-k$ for graded modules, applying this inequality to $V[k]$ rather than $V$ would shift the left side by $-k$ but the right side by $-2k$, leading to the absurd conclusion that $\deg H_p(V)-p\leq \deg H_0(V)+\deg H_1(V)-1-k$ for any $k$. This is impossible unless $\deg H_p(V) = -\infty$ for $p>1$, meaning the ring has homological dimension 1 (as we saw for for $\CC[T]$ above).}   The existence of such a bound for FI-modules reflects the fact that, although a version of the grading shift does exist for FI-modules (see \S\ref{section:shifts-derivatives}), its effect on generators and relations is considerably more complicated. In particular, this shift is not \emph{invertible} for FI-modules.

\para{Infinitely generated FI-modules} One striking feature of \autoref{main:regularity}, and another contrast with polynomial rings, is that its application is not restricted to finitely generated FI-modules: \autoref{main:regularity} bounds the regularity of \emph{any} FI-module which is presented in finite degree.  This is critical for the applications to homology of congruence subgroups in \S\ref{sec:congruence}: for congruence subgroups such as
\[\Gamma_n(t)=\ker\big(\,\GL_n(\mathbb{C}[t]) \to \GL_n(\mathbb{C})\,\big),\]
the FI-modules arising from the homology of $\Gamma_n(t)$ are not even countably generated! Nevertheless, the bounds in  \autoref{thm:congruencestronger} below apply equally well to this case.

\para{Acknowledgements}
The authors are very grateful to the University of Copenhagen for its hospitality in August 2013, where the first version of the main results of this paper were completed. We thank Jennifer Wilson for many helpful comments on an earlier draft of this paper, and Eric Ramos, Jens Reinhold, and Emily Riehl for helpful conversations. We thank Wee Liang Gan and Liping Li for conversations regarding their results in \cite{GanLiCentral} and \cite{GanLi}. We are very grateful to two anonymous referees for their careful reading and thoughtful suggestions, which improved this paper greatly, and especially for comments which led us to the statement of \autoref{thm:saturationalternate} in \S\ref{section:ideal-Im}.

\section{Summary of FI-modules}
\label{section:basics}
In this introductory section, we record the basic definitions and properties of FI-modules that we will use in this paper. Experts who are already familiar with FI-modules can likely skip \S\ref{section:basics} on a first reading (with the exception of \autoref{lem:DaVM} and \autoref{rem:saturationreinterpreted}, which are less standard and play a key role in later sections).

As we mentioned in the introduction, there is a productive analogy between FI-modules and graded $\CC[T]$-modules. For the benefit of readers unfamiliar with FI-modules, in \autoref{table:analogy} we have listed all the  constructions for FI-modules  described in this section, along with the analogous construction for $\CC[T]$-modules. These analogies are not intended as precise mathematical assertions, only as signposts to help the reader orient themself in the world of FI-modules.

\begin{table}
\centering
\begin{tabular}{rl|rl}
  $\FI$\ \ & a category    & $\CC[T]$\ \ &  \multicolumn{1}{l}{an algebra over}
  \\
$\FB\subset \FI$\ \  &its subcat.\ of isos  & $\CC\subset \CC[T]$\ \ &  \multicolumn{1}{l}{a field}
\\
\hline
\multicolumn{1}{l}{\ \ $\iB\colon \FB\into \FI$} & &\multicolumn{1}{l}{$\CC\into \CC[T]$}&
\\ &&\\
$\iB^*\colon \FIMod\to\FBMod$  & forget action of non-isos & $\CTMod\to \CMod$ &take underlying \\
&&&vector space
\\
$M\colon \FBMod\to\FIMod$  &  left adjoint of $\iB^*$ & $\CMod\to \CTMod$ &$V\mapsto \CC[T]\otimes_{\CC} V$
\\
\multicolumn{2}{r|}{$=$ ``free FI-mod on $W$''}&\multicolumn{2}{r}{$=$ free $\CC[T]$-mod on $V$}\\
\hline
\multicolumn{1}{l}{\ \ $\pi\colon \Z\FI\onto \Z\FB$} &  non-isos $\mapsto 0$ & \multicolumn{1}{l}{$\CC[T]\onto \CC$}&$T\mapsto 0$
\\ &&\\
$\pi^*\colon \FBMod\to\FIMod$  &make non-isos act by $0$ & $ \CMod\to \CTMod$ &make $T$ act by $0$ 
\\[2pt]
$\lambda\colon \FIMod\to\FBMod$  &  left adjoint of $\pi^*$ & $\CTMod\to \CMod$ &$M\mapsto M/TM$
\\
&(extend scalars by $\pi$)&&\multicolumn{1}{r}{$=M\otimes_{\CC[T]}\CC$}
\\
$H_0\colon \FIMod\to\FIMod$  &  $H_0\coloneq \pi^*\circ \lambda$ &  \multicolumn{2}{r}{$M/TM$ considered as $\CC[T]$-mod\ \ \ \ \ \ }
\\[2pt]
$\rho\colon \FIMod\to\FBMod$  &  right adjoint of $\pi^*$ & $\CTMod\to \CMod$ &$M\mapsto M[T]$
\\
&&\multicolumn{2}{r}{$=\ker(M\xrightarrow{T}M[1])$}
\\

\hline
$S\colon \FIMod\to \FIMod$&$(SW)_T\coloneq W_{T\disjoint\{\star\}}$&\multicolumn{2}{r}{grading shift $M\mapsto M[1]$\ \qquad\qquad\quad}\\
$D\colon \FIMod\to \FIMod$&$\coker(W\to SW)$&\multicolumn{2}{r}{$\coker(M\xrightarrow{T}M[1])$\qquad\qquad\quad}\\
$K\colon \FIMod\to \FIMod$&$
\ker(W\to SW)=\pi^*\circ \rho$&\multicolumn{2}{r}{$\ker(M\xrightarrow{T}M[1])$\qquad\qquad\quad}\\
\end{tabular}
\caption{Analogies between FI-modules and graded $\CC[T]$-modules}
\label{table:analogy}
  \end{table}
  
(Those readers used to the six-functors formalism may prefer to dualize the right side of \autoref{table:analogy}, thinking of $\iB\colon\FB\into \FI$ as analogous to the structure map $f\colon \Spec\CC[T]\to \Spec\CC$, so that the adjoint functors $M\leftrightarrows \iB^*$  correspond to $f^{-1}\leftrightarrows f_*$. Similarly, $\pi\colon \Z\FI\onto \Z\FB$ is analogous to the closed inclusion $\iC\colon \Spec\CC\to \Spec\CC[T]$,  and the adjunctions $\lambda\leftrightarrows \pi^*\leftrightarrows \rho$ correspond to $\iC^{-1}\leftrightarrows \iC_*=\iC_!\leftrightarrows \iC^!$.)

\subsection{Free FI-modules and generation}
\label{section:free-and-generation}
\para{FB-modules} Just as $\FI$ denotes the category of finite sets and injections, $\FB$ denotes the category of finite sets and bijections. An FB-module $W$ is an element of $\FBMod$, the abelian category of functors $W\colon \FB\to \ZMod$. An FB-module $W$ is just a sequence $W_n$ of $\Z[S_n]$-modules, with no additional structure.

\para{Free FI-modules} 
$\FB$ is the subcategory of $\FI$ consisting of all the isomorphisms (its maximal sub-groupoid). From the inclusion $\iB\colon \FB\into \FI$, we obtain a natural forgetful functor $\iB^*$ from $\FIMod$ to $\FBMod$ that simply forgets about the action of all non-isomorphisms. Its left adjoint $M\colon \FBMod\to \FIMod$ takes an FB-module $W$ to the ``free FI-module $M(W)$ on $W$''. We call any FI-module of the form $M(W)$ a \emph{free FI-module}.

We recall from \cite[Definition~2.2.2]{CEF} an explicit formula for $M(W)$, from which we can see that $M$ is exact.
\begin{equation}
\label{eq:MW}
M(W)_T=\bigoplus_{S\subset T}W_S
\end{equation}
For notational convenience, for $m\in \N$ we write $M(m)\coloneq M(\Z[S_m])$. These FI-modules have the defining property that $\Hom_{\FIMod}(M(m),V)\iso V_m$, since we can write $M(m)\iso\Z[\Hom_{\FI}([m],-)]$.

As a consequence, $M(m)$ is a projective FI-module (they are the ``principal projective'' FI-modules). In general, an FI-module is projective if and only if it is a summand of some $\bigoplus_{i\in I}M(m_i)$. 

We point out that despite the name, free FI-modules need not be projective (since non-projective $\Z[S_n]$-modules are in abundance!).\footnote{For example, consider the FI-module $V$ for which  $V_T\coloneq \Z[e_{\{i,j\}}]_{i\neq j\in T}$, the free abelian group  on the 2-element subsets of $T$. This FI-module $V$ is free on the FB-module $W$ having $W_2\iso \Z$ (the trivial $\Z[S_2]$-module) and $W_n=0$ for $n\neq 2$.
However $V$ is not projective, since $W_2$ is not a projective $\Z[S_2]$-module.} Nevertheless, for our purposes free FI-modules will be just as good as projective FI-modules (see 
\autoref{lemma:freeHacyclic} and \autoref{cor:freeacyclic}), so this discrepancy will not bother us.

\para{Generation in degree $\leq k$}
Every FI-module $V$ has a natural increasing filtration \[V_{\langle\leq 0\rangle}\subset V_{\langle\leq 1\rangle}\subset \cdots \subset V_{\langle\leq m\rangle}\subset \cdots\subset V= \bigcup_{m\geq 0} V_{\langle\leq m\rangle}\]
where $V_{\langle\leq m\rangle}$ is the sub-FI-module of $V$ ``generated by elements in degree $\leq m$''. This filtration, which  is respected by all maps of FI-modules, can be defined as follows.

Given an FI-module $V$, by a slight abuse of notation we write $M(V)$ for the free FI-module on the  FB-module $\iB^*V$ underlying $V$. 
From the adjunction $M\leftrightarrows \iB^*$ we have a canonical map $M(V)\onto V$, which is always surjective. 
We modify this slightly to define the filtration $V_{\langle \leq m\rangle}$.

\begin{definition}
Let $V_{\leq m}$ be the FB-module defined by $(V_{\leq m})_T=V_T$ if $\abs{T}\leq m$ and $(V_{\leq m})_T=0$ if $\abs{T}> m$. Then the natural inclusion of FB-modules $V_{\leq m}\into \iB^*V$ induces a  map of FI-modules $M(V_{\leq m})\to V$.

We define $V_{\langle \leq m\rangle}\subset V$ to be the image of the canonical map $M(V_{\leq m})\to V$.
Equivalently, $V_{\langle \leq m\rangle}$ is the smallest sub-FI-module $U\subset V$ satisfying $U_n=V_n$ for all $n\leq m$. 
We sometimes write $V_{\langle <m\rangle}$ as an abbreviation for $V_{\langle \leq m-1\rangle}$.
\end{definition} 

In the introduction we said that an FI-module $W$ is \emph{generated in degree $\leq m$} if $\deg H_0(W)\leq m$, but there are many equivalent ways to formulate this definition.
\pagebreak
\begin{lemma}
\label{lem:genequiv}
Let $V$ be an FI-module, and fix $m\geq 0$. The following are equivalent:
\begin{enumerate}[label=(\roman*)]
\item $V$ is generated in degree $\leq m$.
\item $\deg H_0(V)\leq m$.
\item $V=V_{\langle \leq m\rangle}$.
\item $V$ admits a surjection from $\bigoplus_{i\in I}M(m_i)$ with all $m_i\leq m$.
\item The natural map $M(W_m)\to W$ is surjective in degrees $\geq m$.
\end{enumerate}
\end{lemma}

\para{The functor $H_0$}
In the other direction, we do not quite have a projection from $\FI$ to $\FB$, because the non-invertible morphisms in $\FI$ have nowhere to go.  One wants to map them to zero, but there's no ``zero morphism" in $\FB$.  This problem can be solved by passing to the $\Z$-enriched versions $\Z\FI$ and $\Z\FB$ of these categories:  we now have a functor $\pi\colon\Z\FI\to\Z\FB$ which sends all non-invertible morphisms to zero and is the identity on isomorphisms. 

This induces the ``extension by zero'' functor $\pi^*\colon \FBMod\to \FIMod$ which takes an FB-module $W$ and simply regards it as an FI-module by defining $f_*=0$ for all non-invertible $f\colon S\to T$. This functor is exact and has both a left adjoint $\lambda$ and right adjoint $\rho$.

In this section, we consider the left adjoint $\lambda\colon \FIMod\to \FBMod$. Since every noninvertible map $f\colon S\into T$ increases cardinality, we have the formula $(\lambda V)_n=(V/V_{\langle <n\rangle})_n$. This is almost exactly the definition of $H_0\colon \FIMod\to \FIMod$ given in the introduction; the only difference is that $\lambda V$ is an FB-module whereas $H_0(V)$ is the same thing regarded as an FI-module, i.e.\  $H_0=\pi^*\circ \lambda$.

We adopt the convention in this paper that if $F$ is a right-exact functor, $H_p^F$ denotes its $p$-th left-derived functor. As we explained in the introduction, we write $H_p(V)$ for the derived functors $H_p^{H_0}(V)$ of $H_0$, and call these the \emph{FI-homology} of $V$.

\begin{lemma}
\label{lemma:freeHacyclic}
Free FI-modules are $H_0$-acyclic.
\end{lemma}
\begin{proof}
Our goal is to prove that $H_p(M(W))=0$ for $p>0$. Since $M$ is exact and takes projectives to projectives, there is an isomorphism $H_p(M(W))\iso H_p^{H_0\circ M}(W)$.

However, the composition $H_0\circ M$ is just the exact functor $\pi^*$. Indeed, the composition $\pi\circ \iB$ is the identity,
so $\iB^*\circ \pi^*=\id$. It follows that its left adjoint $\lambda \circ M$ is the identity as well. Since $H_0=\pi^*\circ \lambda$, we have $H_0\circ M=\pi^*\circ \lambda\circ M= \pi^*$ as claimed. We conclude that $H_p(M(W))\iso H_p^{H_0\circ M}(W)=H_p^{\pi^*}(W)$, which vanishes for $p>0$ since $\pi^*$ is exact.
\end{proof}

We can now explain how \autoref{maincor:kernel} follows from \autoref{main:regularity}.
\begin{proof}[Proof of \autoref{maincor:kernel}]
If $M=0$ the corollary is trivial, so assume that $M\neq 0$. Let $K=\ker(V\to M)$ and $W=\coker(V\to M)$. Thanks to the equivalences in \autoref{lem:genequiv}, the statement of the corollary is that $\deg H_0(K)\leq \deg H_0(M)+\deg H_0(V)+1$.

From the exact sequence $0\to K\to V\to M\to W\to 0$ we obtain the inequalities
\[
\begin{aligned}\deg H_0(W)&\leq \phantom{\max\big(\,\deg H_0(V),\ \ }\deg H_0(M)\\
\deg H_1(W)&\leq \max\big(\,\deg H_0(V),\ \ \deg H_1(M)\phantom{,\ \ \deg H_2(W)}\,\big)\\
\deg H_0(K)&\leq \max\big(\,\deg H_0(V),\ \ \deg H_1(M),\ \ \deg H_2(W)\,\big)\\
\end{aligned}
\]
Therefore to prove the corollary, it suffices to show that the degrees of $H_0(V)$, $H_1(M)$, and $H_2(W)$ are bounded by $\deg H_0(M)+\deg H_0(V)+1$. For $H_0(V)$ this is trivial, since $\deg H_0(M)\geq 0$. Since $M$ is free, $H_1(M)=0$ by \autoref{lemma:freeHacyclic}, so $\deg H_1(M)=-\infty$; this also shows $\deg H_1(W)\leq \deg H_0(V)$. Finally, applying \autoref{main:regularity} to $W$ shows that
\[\deg H_2(W)\leq \deg H_0(W)+\deg H_1(W)+1\leq \deg H_0(M)+\deg H_0(V)+1,\] as desired.
\end{proof}

\subsection{Shifts and derivatives of FI-modules}
\label{section:shifts-derivatives}
\para{The shift functor $S$}
Fix a one-element set $\{\star\}$. Let $\disjoint\colon \Sets\times \Sets\to \Sets$ be the coproduct, i.e.\ the disjoint union of sets. This must be formalized in some fixed functorial way such as $S\disjoint T\coloneq (S\times \{0\})\cup (T\times \{1\})$; but since the coproduct is unique up to canonical isomorphism, the choice of formalization is irrelevant.

The disjoint union with $\{\star\}$ defines a functor $\sigma\colon \FI\to \FI$ by $T\mapsto T\disjoint \{\star\}$. The \emph{shift functor} $S\colon \FIMod\to \FIMod$ is given by precomposition with $\sigma$: the FI-module $SV$ is the composition $SV\colon \FI\xrightarrow{\sigma}\FI\xrightarrow{V} \ZMod$. Concretely, for any finite set $T$ we have $(SV)_T=V_{T\disjoint \{\star\}}$. The functor $S$ is evidently exact.

\para{The kernel functor $K$ and derivative functor $D$}
The inclusion of $S$ into $S\disjoint \{\star\}$ defines a natural transformation from $\id_{\FI}$ to $\sigma$. From this we obtain a natural transformation $\iSV$ from $\id_{\FIMod}$ to $S$. Concretely, this is a natural map of FI-modules $\iSV\colon V\to SV$ which, for every finite set $T$, sends $V_T$ to $(SV)_T=V_{T\disjoint \{\star\}}$ via the map correspoding to the inclusion $\iset_T$ of $T$ into $T\disjoint \{\star\}$.

The functor $D\colon \FIMod\to \FIMod$, the \emph{derivative}, is defined to be the cokernel of this map:
\[DV\coloneq \coker(V\xrightarrow{\iSV} SV).\]
We similarly define $K\colon \FIMod\to \FIMod$ to be the kernel: $KV\coloneq \ker(V\xrightarrow{\iSV} SV)$.
For any FI-module, we have a natural exact sequence
\[0\to KV\to V\to SV\to DV\to 0.\]
Since $\id$ and $S$ are exact functors, $D$ is right-exact and $K$ is left exact. Concretely, we have \[(DV)_T\iso V_{T\disjoint \{\star\}}/\im(V_T\to V_{T\disjoint \{\star\}})\qquad\text{and}\qquad
(KV)_T=\{v\in V_T\,|\,\iset(v)=0\in V_{T\disjoint\{\star\}}\}.\]
From this formula for $KV$, one can check that the functor $K$ essentially coincides with the right adjoint $\rho\colon \FIMod\to \FBMod$ of $\pi^*$; as we saw with $H_0$, the only difference is that $KV$ is $\rho V$ considered as an FI-module, i.e.\ $K=\pi^*\circ \rho$.
\begin{numberedremark}
Readers paying attention to the analogies in \autoref{table:analogy} might object that although $H_0$ and $D$ are very different functors of FI-modules, the table indicates that both correspond to the functor $M\mapsto M/TM$ of graded $\CC[T]$-modules. But this is not quite right, and by being careful with gradings we can see the distinction: $H_0$ corresponds to 
$\coker(M[-1]\xrightarrow{T} M)$ whereas $D$ corresponds to 
$\coker(M\xrightarrow{T} M[1])$. In the graded case we rarely need to worry about the distinction, since grading shifts are invertible. But for FI-modules this is not true, and the distinction is important. (\autoref{lemma:DMMS} below may clarify the behavior of $D$.)
\end{numberedremark}
\pagebreak

\begin{lemma}
\label{lem:torsion-freeK}
An FI-module $V$ is torsion-free if and only if $KV=0$.
\end{lemma}
\begin{proof}
Recall that an FI-module $V$ is \emph{torsion-free} if for any injection $f\colon S\into T$ of finite sets, the map $f_*\colon V_S\to V_T$ is injective. By a simple induction, this holds if and only if $f_*$ is injective for all $f\colon S\into T$ with $\abs{T}=\abs{S}+1$. However such an inclusion can be factored as $f=g\circ \iset_S$ for some bijection $g\colon S\disjoint\{\star\}\iso T$. Since $g_*$ is necessarily injective, we see that $V$ is torsion-free if and only if $\iSV_S=(\iset_S)_*\colon V_S\to V_{S\disjoint\{\star\}}$ is injective for all finite sets $S$, i.e.\ if $KV=0$.
\end{proof}

\para{Iterates of shift and derivative} We can iterate the shift functor $S$, obtaining FI-modules $S^bV$ for any $b\geq 0$. To avoid the notational confusion of writing $(S^2V)_T\iso V_{T\disjoint\{\star\}\disjoint\{\star\}}$, we adopt the notation that $[\star_b]$ denotes a fixed $b$-element set $[\star_b]\coloneq \{\star_1,\ldots,\star_b\}$. We can then naturally identify $(S^2V)_T\iso V_{T\disjoint [\star_2]}$, and so on.

The iterates $D^a$ are also right-exact and can be described quite explicitly.  For every FI-module V and every finite set $T$, we have
\begin{equation}
\label{eq:dkdisc}
(D^a V)_T \iso \frac{V_{T \disjoint [\star_a]}}{\sum_{j=1}^a \im(V_{T \disjoint [\star_a]-\{\star_j\}})}
\end{equation}
where $\im(V_{T \disjoint [\star_a]-\{\star_j\}})$ denotes the image of the natural map $V_{T \disjoint [\star_a]-\{\star_j\}}\to V_{T\disjoint[\star_a]}$ induced by the inclusion $T \disjoint [\star_a]-\{\star_j\}\subset T \disjoint [\star_a]$.
We remark that $D^a$ is the left adjoint of the functor $B^a$ of \cite[Definition~2.16]{CEFN}.

For any submodule $V\subset M$ the inclusion induces a map $D^a V\to D^a M$.
\begin{lemma}
\label{lem:DaVM}
If $M$ is a torsion-free FI-module and $V\subset M$ a submodule, 
\[\ker(D^a V\to D^a M)_{n-a}=0\ \ \iff\ \  V_n\cap \big(M_{[n]-\{1\}}+\cdots+M_{[n]-\{a\}}\big)=V_{[n]-\{1\}}+\cdots+V_{[n]-\{a\}}.\]
\end{lemma}
\begin{proof} Since $M$ is torsion-free we can identify $M_S$ with its image in $M_{S\disjoint \{\star\}}$ and so on, so \[\ker(D^a V\to D^a M)_T\iso\ker\left(\frac{V_{T \disjoint [\star_a]}}{\sum_{j=1}^a V_{T \disjoint [\star_a]-\{\star_j\}}}\to  \frac{M_{T \disjoint [\star_a]}}{\sum_{j=1}^a M_{T \disjoint [\star_a]-\{\star_j\}}}\right).\]
In other words, 
\[\ker(D^a V\to D^a M)_T=0\qquad\iff\qquad V_{T\disjoint [\star_a]}\cap\left(\sum_{j=1}^a M_{T \disjoint [\star_a]-\{\star_j\}}\right)=\sum_{j=1}^a V_{T \disjoint [\star_a]-\{\star_j\}}.\]
Setting $T=\{a+1,\ldots,n\}$ and identifying $[\star_a]$ with $\{1,\ldots,a\}$, we obtain the desired expression.
\end{proof}
\begin{numberedremark}
\label{rem:saturationreinterpreted}
Notice that the right side of \autoref{lem:DaVM} is precisely the conclusion of \autoref{main:saturation}.
Therefore we can restate \autoref{main:saturation} as saying: if $M$ is torsion-free and generated in degree $\leq k$ and $V\subset M$ is generated in degree $\leq d$, then $\ker(D^a V\to D^a M)_{n-a}$ vanishes for $n>d+\min(k,d)$; in other words, $\deg \ker(D^a V\to D^a M)\leq d+\min(k,d)-a$. This observation will be used in Section~\ref{section:bounds-and-proof-of-thm-A}, and specifically in the proof of \autoref{thm:Dkhomology} to obtain bounds on $H_p^{D^a}$.
\end{numberedremark}

\section{Combinatorics of finite injections and FI-modules}
\label{sec:combinatorics}

The goal of this section is to prove \autoref{main:saturation}. In \S\ref{section:ideal-Im} we generalize \autoref{main:saturation} to \autoref{thm:saturationalternate} which does not refer to the ambient FI-module $M$, and is of independent interest. In \S\ref{section:combinatorics} we establish the combinatorial properties of $\Z[\Hom_{\FI}([d],[n])]$ that make our proof possible; throughout that section we do not mention FI-modules at all. In \S\ref{section:saturationproof} we apply these properties to prove \autoref{thm:saturationalternate}. But before moving to the combinatorics, we begin by motivating the connections with FI-modules.

\subsection{The ideal \texorpdfstring{$I_m$ and \autoref{thm:saturationalternate}}{I-m and Theorem E'}}
\label{section:ideal-Im}
\para{The ideal $I_m$}
For each pair of distinct elements $i\neq j$ in $[n]$, we write $(i\ \,j)$ for the transposition in $S_n$ interchanging $i$ and $j$, and we define $J^i_j \coloneq  \id - (i\ \,j)\in \Z[S_n]$. Note that $J^i_j=J^j_i$, and that $J^i_j$ and $J^k_l$ commute when their four indices are distinct (since the transpositions $(i\ \,j)$ and $(k\ \,l)$ commute in this case). 

For  $m\in \N$, define $I_m\subset \Z[S_n]$ to be the two-sided ideal generated by products of the form
\begin{displaymath}
J^{i_1}_{j_1}J^{i_2}_{j_2} \ldots J^{i_m}_{j_m}
\end{displaymath}
where $i_1, j_1, \ldots, i_m, j_m$ are $2m$ \emph{distinct} elements of $[n]$.  (In particular, the terms of the product commute.) Multiplying out such a product, we have 
\begin{equation}
\label{eq:JSexpansion}
J^{i_1}_{j_1}J^{i_2}_{j_2} \ldots J^{i_m}_{j_m}=
\sum_{K\subset [m]} (-1)^{\abs{K}} \prod_{k\in K}(i_k\ \,j_k)=\sum_{\sigma\in (\Z/2)^m} (-1)^\sigma \sigma,
\end{equation}
where $(\Z/2)^m$ denotes the subgroup generated by the commuting transpositions $(i_k\ \,j_k)$, and $(-1)^\sigma$ denotes the image of $\sigma$ under the sign homomorphism $S_n \ra \pm 1$.

Although the ideals $I_m$ will play multiple different roles in the proof of \autoref{main:saturation}, the following property provides a simple illustration of why we consider these ideals. 
Recall that the group ring $\Z[S_n]$ acts on $W_n$ for any FI-module $W$.  
\begin{prop}
\label{pr:bigsat}
Let $M$ be an FI-module generated in degree $\leq k$. Then $I_{k+1}\cdot M_n = 0$ for all $n\geq 0$.
\end{prop}
\begin{proof}
We prove first that $I_m$ annihilates the free module $M(a)$  if $a<m$, meaning that $I_m\cdot M(a)_n=0$ for all $n$.
For any $a$, a basis for $M(a)_n$ is given by injections $f\colon [a]\into [n]$. The key observation is that given $f\colon [a]\into [n]$ and a generator $J^{i_1}_{j_1}J^{i_2}_{j_2} \ldots J^{i_{m}}_{j_{m}} \in I_{m}$, 
\begin{equation}
\label{eq:bigsatlemma}
\text{ if $\im f\cap \{i_\ell,j_\ell\}=\emptyset$ for some $\ell\in [m]$,\qquad then $J^{i_1}_{j_1}J^{i_2}_{j_2} \ldots J^{i_m}_{j_m}\cdot f=0$.}
\end{equation}
Indeed the assumption implies $(i_\ell\ \,j_\ell)\circ f=f$, so $J^{i_\ell}_{j_\ell}=\id - (i_\ell\ \,j_\ell)$ satisfies $J^{i_\ell}_{j_\ell}\cdot f=0$. Since the terms of the product commute, it follows that $J^{i_1}_{j_1}J^{i_2}_{j_2} \ldots J^{i_m}_{j_m}\cdot f=0$.

However when $a<m$, for \emph{every} $f\colon [a]\into [m]$ and $J=J^{i_1}_{j_1}J^{i_2}_{j_2} \ldots J^{i_{m}}_{j_{m}} \in I_{m}$ there exists some $\ell$ for which $\im f\cap \{i_\ell,j_\ell\}=\emptyset$. Therefore $J\cdot f=0$ for all basis elements $f\in M(a)_n$ and all generators $J\in I_m$, proving that $I_m\cdot M(a)=0$ as claimed.

Returning to the general claim, let $M$ be an FI-module generated in degree $\leq k$.
By \autoref{lem:genequiv}(iv), $M$ is a quotient of a sum of free modules $M(a)$ generated in degrees $a\leq k$. We have just proved that $I_{k+1}$ annihilates any such free module, so it annihilates the quotient $M$ as well. 
\end{proof}

\para{Generalizing \autoref{main:saturation} by removing $M$}
The statement of \autoref{main:saturation} can be generalized by removing $M$ from its statement. Recall that \autoref{main:saturation} states that if $M$ is a torsion-free FI-module and $V\subset M$ is a submodule, then\begin{equation}
\label{eq:thmE}
V_n\cap \big(M_{[n]-\{1\}}+\cdots+M_{[n]-\{a\}}\big)=V_{[n]-\{1\}}+\cdots+V_{[n]-\{a\}}
\end{equation} for sufficiently large $n$. Though it may not be obvious, the central object in this statement is $V$. In fact, we can remove the FI-module $M$ from the statement entirely, and at the same time strengthen the theorem.

Consider the case $a=1$, when our goal \eqref{eq:thmE} is that $V_n\cap M_{[n]-\{1\}}$ coincides with $V_{[n]-\{1\}}$ for large enough $n$. When $M$ is free, the submodule $M_{[n]-\{1\}}\subset M_{[n]}$ can be cut out as \[M_{[n]-\{1\}}=\big\{m\in M_{[n]}\subset M_{[n]\disjoint [\star_1]}\,\big\vert\, (1\ \,\star_1)\cdot m = m\,\big\}.\]

In other words, recalling that $\incThmD\colon [n]\into [n]\disjoint [\star_1]$ denotes the standard inclusion, the element $\J_{[1]}\coloneq J_1^{\star_1}\circ \incThmD\in \Z[\Hom_{\FI}([n],[n]\disjoint [\star_1])]$ has the property that \[M_{[n]-\{1\}}=\ker \J_{[1]}|_{M_{[n]}}\] when $M$ is free. 
For general $M$ we need \emph{not} have equality here, but we do always have the containment $M_{[n]-\{1\}}\subset \ker \J_{[1]}|_{M_{[n]}}$. Intersecting with $V$, we always have the containments
\[V_{[n]-\{1\}}\ \subset\ 
V_n\cap M_{[n]-\{1\}}\ \subset\ 
\ker \J_{[1]}|_{V_{[n]}}.\] 
The statement of \autoref{main:saturation} is that the \emph{first} containment is an equality for large enough $n$. But we can actually prove the stronger statement that \emph{both} are equalities: $V_{[n]-\{1\}}=\ker \J_{[1]}|_{V_{[n]}}$ for large enough $n$. Notice that this statement no longer makes reference to $M$! 

For larger $a$, we consider the element $\J_{[a]}\coloneq J_1^{\star_1}\cdots J_a^{\star_a}\circ \incThmD\in \Z[\Hom_{\FI}([n],[n]\disjoint [\star_a])]$. We saw in the previous paragraph that $M_{[n]-\{i\}}\subset \ker J_i^{\star_i}$ (identifying $M_{[n]-\{i\}}$ with its image). Since all the operators $J_i^{\star_i}$ commute, this implies that $M_{[n]-\{i\}}\subset \ker \J_{[a]}$ for any $i\in [a]$, so
\[M_{[n]-\{1\}}+\cdots+M_{[n]-\{a\}}\subset \ker \J_{[a]}|_{M_{[n]}}.\]
This means that
 \begin{equation}
 \label{eq:VMJ}
 V_{[n]-\{1\}}+\cdots+V_{[n]-\{a\}}
\ \subset\ 
V_n\cap \big(M_{[n]-\{1\}}+\cdots+M_{[n]-\{a\}}\big)
 \ \subset\ 
 \ker \J_{[a]}|_{V_{[n]}}
 \end{equation} for any $n\geq a$. 
This leads us to the following generalization of \autoref{main:saturation}.

\begin{theorem-prime}{main:saturation}
\label{thm:saturationalternate}
Let $V$ be a torsion-free FI-module generated in degree $\leq d$ satisfying $I_{K+1}\cdot V=0$. Then for all $n> K+d$ and any $a\leq n$,
\[V_{[n]-\{1\}}+\cdots+V_{[n]-\{a\}}=\ker \J_{[a]}.\]
\end{theorem-prime}

\autoref{thm:saturationalternate} is proved in  \S\ref{section:saturationproof} below, but we first verify here that it implies \autoref{main:saturation}.

\begin{proof}[Proof that \autoref{thm:saturationalternate} implies \autoref{main:saturation}]
We begin in the setup of \autoref{main:saturation}, so let $M$ be a torsion-free FI-module generated in degree $\leq k$, and let $V$ be a submodule of $M$ generated in degree $\leq d$.

\autoref{pr:bigsat} states that $I_{k+1}\cdot M=0$, so the same is true of its submodule $V$. Applying \autoref{pr:bigsat} to $V$ directly shows that $I_{d+1}\cdot V=0$. Therefore if we set  $K=\min(k,d)$, we have $I_{K+1}\cdot V=0$.

Applying \autoref{thm:saturationalternate}, we conclude that $V_{[n]-\{1\}}+\cdots+V_{[n]-\{a\}}=
 \ker \J_{[a]}|_{V_{[n]}}$ for all $n>K+d$. In light of \eqref{eq:VMJ}, this implies that $V_{[n]-\{1\}}+\cdots+V_{[n]-\{a\}}=
V_n\cap \big(M_{[n]-\{1\}}+\cdots+M_{[n]-\{a\}}\big)$ for $n>K+d$, as desired.
\end{proof}

In fact, \autoref{thm:saturationalternate} is strictly stronger than \autoref{main:saturation}. To see this, notice that \autoref{main:saturation} says nothing when $V=M$, while \autoref{thm:saturationalternate} implies the following  structural statement (by taking $a=K+1$ and noting that $\J_{[K+1]}\in I_{K+1}$), which is nontrivial whenever $K<d$.

\begin{corollary}
\label{cor:VKK}
Let $V$ be a torsion-free FI-module generated in degree $\leq d$ satisfying $I_{K+1}\cdot V=0$. (For example, this holds if $V$ can be embedded into some FI-module generated in degree $\leq K$.) Then for all $n> K+d$,
\[V_n = V_{[n]-\{1\}}+\cdots+V_{[n]-\{K\}}+V_{[n]-\{K+1\}}.\]
\end{corollary}

\subsection{The combinatorics of \texorpdfstring{$\Z[\Hom_{\FI}([d],[n])]$}{Z[Hom-FI([d],[n])]}}
\label{section:combinatorics}

The discussion above did not depend on any ordering on $[n]$ (essentially treating it as an arbitrary finite set). By contrast, throughout the rest of this section we rely heavily on the ordering on $[n]$. This is inconsistent with the philosophy of FI-modules, so throughout \S\ref{section:combinatorics} we will not mention the category $\FI$ at all.

\begin{definition}[The collection $\Sigma(b)$]
For  $b\in \N$, let $\Sigma(b)$ denote the set of $b$-element subsets $S\subset [2b]$ satisfying the following property:
\begin{equation}
\label{catlabel}\tag{\cattag}
\text{The }i\text{-th largest element of }S\text{ is at most } 2i-1.
\end{equation}
For  $a\in \N$ with $1\leq a\leq b$, let $\Sigma(a,b)\subset \Sigma(b)$ consist of all those $S\in \Sigma(b)$ containing $[a]$:
\begin{equation}
\label{eq:Sab}
\Sigma(a,b)\coloneq \big\{S\in \Sigma(b)\,\big|\,[a]\subset S\subset [2b]\,\big\}
\end{equation}
\end{definition}
For example, it follows from (\cat) that $1\in S$ for any $S\in \Sigma(b)$, so for any $b\in \N$ we have $\Sigma(1,b)=\Sigma(b)$.
At the other extreme, we have $\Sigma(b,b)=\{[b]\}$. The subsets $\Sigma(a,b)$ interpolate between $\Sigma(1,b)=\Sigma(b)$ and $\Sigma(b,b)=\{[b]\}$; for example:
\begin{align*}
\Sigma(1,4)=&&&1234, 1235, 1236, 1237, 1245,1246, 1247, 
1256, 1257,
1345,1346,1347,
1356,1357\\
\Sigma(2,4)=&&&1234, 1235, 1236, 1237, 1245,1246, 1247, 
1256, 1257\\
\Sigma(3,4)=&&&1234, 1235, 1236, 1237\\
\Sigma(4,4)=&&&1234\end{align*}
We have written the elements of $\Sigma(a,b)$ in lexicographic order, which ordering we denote by $\preccurlyeq$.
We denote by $\comp{S}$ the complement $\comp{S}\coloneq [2b]\setminus S$. We will only use this notation for $b$-element subsets $S\subset [2b]$, so the notation is unambiguous; in particular, $\comp{S}$ is always a $b$-element subset of $[2b]$ as well.

\begin{numberedremark}
\label{remark:Sigmab}
We record some relations between the different collections $\Sigma(a,b)$.
\begin{enumerate}[label={(\alph*)},wide]
\item \label{part:Scupm} For any $S \in \Sigma(b)$ and any $m\leq 2b+1$ with $m\notin S$, the union $S \cup \set{m}$ belongs to $\Sigma(b+1)$. In particular, this holds if $m\in \comp{S}$. If $S\in \Sigma(a,b)$, then $S\cup\set{m}\in \Sigma(a,b+1)$.
\item \label{part:acinab} For any $S\in \Sigma(b)$ and any $c\leq b$, if $R\subset S$ is the $c$-element subset consisting of the $c$ smallest elements, then $R\in \Sigma(c)$. If $S\in \Sigma(a,b)$ for $a\leq c\leq b$, then $R\in \Sigma(a,c)$ as well.
\item If $S,T\in \Sigma(b)$ satisfy $T\preccurlyeq S$ and $S\in \Sigma(a,b)$, then $T\in \Sigma(a,b)$ as well. In other words, $\Sigma(a,b)$ is an ``initial segment'' of $\Sigma(b)$ (this is immediately visible in the description of $\Sigma(a,4)$ above).
\end{enumerate}
\end{numberedremark}

\para{Descendants}
The condition (\cat) gives one way to define the Catalan numbers: the $n$-th Catalan number is $\abs{\Sigma(n)} = \frac{1}{n+1}\binom{2n}{n}$. This is not a coincidence; our interest in $\Sigma(b)$ comes from the following characterization of the sets $S\in \Sigma(b)$, which is related to another definition of the Catalan numbers.

Given any $b$-element subset $S\subset [2b]$, write the elements of $S$ in increasing order as $s_1, \ldots, s_b$ and the elements of $\comp{S}$ in increasing order as $t_1, \ldots, t_b$.  Let  $(\Z/2)^S$ denote the subgroup of $S_{2b}$ generated by the commuting transpositions $(s_k\ \,t_k)\in S_{2b}$. If we define $J_S\in I_b$ as
\begin{equation}
\label{eq:defJS}
J_S \coloneq \prod_i J_{s_i}^{t_i},
\end{equation}
by \eqref{eq:JSexpansion} we have $J_S=\sum_{\sigma\in (\Z/2)^S} (-1)^\sigma \sigma$. In these terms, the defining property (\cat) of $\Sigma(b)$ has the following formulation:
\begin{equation}
\label{eq:lexfirst}
S\in \Sigma(b)\quad\iff\quad S \text{ is lexicographically first among } \{\sigma\cdot S\,|\,\sigma\in (\Z/2)^S\}
\end{equation}
Given $S\in \Sigma(b)$, we refer to the subsets $\{\sigma\cdot S\,|\,\sigma\in (\Z/2)^S\}$ as the \emph{descendants} of $S$; by \eqref{eq:lexfirst}, $S$ lexicographically precedes all of its descendants.\footnote{A set $S$ and its descendant $\sigma\cdot S$ need not determine the same subgroup $(\Z/2)^S\neq (\Z/2)^{\sigma\cdot S}$, so the relation of being a descendant is neither symmetric nor transitive. For example, if $S=\{1,2\}\subset [4]$, then $S'=\{1,4\}$ is a descendant of $S$, but $(\Z/2)^S=\langle (1\ \,3),(2\ \,4)\rangle$ whereas $(\Z/2)^{S'}=\langle (1\ \,2),(3\ \,4)\rangle$.
The descendants of $S$ are $S=12$, $S'=14$, $23$, and $34$ whereas the descendants of $S'$ are $S'=14$, $23$, $13$, and $24$.
}
In fact, we will use the following generalization. For any subset $U\subset [n]$ with $S\subset U$, and any $b$ distinct elements $u_1<\cdots<u_b$ of $[n]\setminus U$, we can consider the subgroup $(\Z/2)^b$ generated by the disjoint transpositions $(s_i\ \,u_i)$. By comparison with \eqref{eq:lexfirst}, it is straightforward to conclude:
\begin{lemma}
\label{lem:lexfirstarbitrary}
$S\in \Sigma(b)\quad\implies\quad U \text{ is lexicographically first among } \{\sigma\cdot U\,|\,\sigma\in (\Z/2)^b\}$.
\end{lemma}

\para{The $S_n$-module $F$ and subgroups}
Fix $d\in \N$ and $n\in \N$ for the remainder of Section~\ref{sec:combinatorics}. Let $F$ denote the $\Z[S_n]$-module associated to the permutation action on the set of injections $f\colon [d]\into [n]$. (In other words, as an $S_n$-module $F$ is isomorphic to $\Z[\Hom_{\FI}([d],[n])]$; however, we will wait until the next section to explore this connection with the category $\FI$.)

\begin{definition}
\label{def:Fsubgroups}
We  define certain subgroups of the free abelian group $F$ corresponding to particular subsets of the basis  $\{f\colon [d]\into [n]\}$. In these definitions $S$ is a $b$-element subset $S\in\Sigma(b)$.
\begin{align*}
F^{\neq S}&\coloneq &&\,\big\langle f\colon [d]\into [n]\,\big\vert\,S\not\subset \im f\big\rangle\\
F^{b}&\coloneq \bigcap_{S\in \Sigma(b)} F^{\neq S}&=&\,\big\langle f\colon [d]\into [n]\,\big\vert\,\forall S\in \Sigma(b), S\not\subset \im f\big\rangle\\
F^{a,b}&\coloneq \bigcap_{S\in \Sigma(a,b)} F^{\neq S}&=&\,\big\langle f\colon [d]\into [n]\,\big\vert\,\forall S\in \Sigma(a,b), S\not\subset \im f\big\rangle\\
F_{=S}&\coloneq &&\,\big\langle f\colon [d]\into [n]\,\big\vert\,\im f\cap [2b]=S\big\rangle
\end{align*}
In general none of these subgroups are preserved by the action of $S_n$ on $F$.
\end{definition}

We emphasize the contrast between $F^{\neq S}$ and $F_{=S}$: for fixed $b$, a given injection $f\colon [d]\into [n]$ may lie in  $F^{\neq S}$ for many different $S\in \Sigma(b)$; in contrast,  $f$ lies in $F_{=S}$ for at most one $S\in \Sigma(b)$ (namely $S=\im f\cap [2b]$, if this subset happens to belong to $\Sigma(b)$).

Since $\Sigma(b)=\Sigma(1,b)\supset \cdots \supset \Sigma(b,b)$, we have $F^b=F^{1,b}\subset F^{2,b}\subset \cdots\subset F^{b,b}$. 
Similarly, from \autoref{remark:Sigmab}\ref{part:acinab} we have $F^{a,a}\subset \cdots\subset F^{a,b}\subset\cdots$. In other words, if $a\leq a'$ and $b\leq b'$ then $F^{a,b}\subset F^{a',b'}$.  Note that, since $\Sigma(a,a)$ consists of the single set $S=[a]$, the subgroup $F^{a,a}$ is spanned by injections $f\colon [d]\into [n]$ with $i\notin \im f$ for some $i\in [a]$.

We make no assumptions whatsoever on $d$, $n$, or $b$ in this section, although in some cases the definitions become rather trivial. (For example, when $b>d$ we have $F=F^b$; when $d>n$ we have $F=0$; when $2b>n$ we have $I_b=0$.)

\begin{prop} 
\label{pr:bigb} For any $b$ such that $n \geq b+d$ we have
\begin{displaymath}
F=I_b \cdot F + F^b.
\end{displaymath}
\end{prop}

\begin{proof} It is vacuous that $I_b\cdot F+F^b\subset F$, so we must prove that $F\subset I_b\cdot F+F^b$. Assume otherwise; then some basis element $f$ does not lie in $I_b\cdot F+F^b$. Choose $f$ so that $\im f$ is lexicographically last among all such $f$. Since $f\notin F^b$, there exists some $S\in \Sigma(b)$ with $S\subset \im f$. Since $n\geq b+d$, we may choose $b$ distinct elements $u_1<\cdots<u_b$ from $[n]\setminus \im f$. Let $J=J_{s_1}^{u_1}\cdots J_{s_b}^{u_b}$, and consider the element \[(J-\id)\cdot f=\sum_{\sigma\neq 1\in (\Z/2)^b}(-1)^\sigma \sigma\cdot f.\] By 
\autoref{lem:lexfirstarbitrary} we have $\im(\sigma\cdot f)=\sigma\cdot \im f\succ \im f$ for all $\sigma\neq 1$. By our definition of $f$ (that its image was lexicographically last), $\sigma\cdot f$ is contained in $I_b\cdot F+F^b$ for all $\sigma\neq 1$, so $(J-\id)\cdot f\in I_b\cdot F+F^b$. However $J\cdot f\in I_b\cdot F$ by definition, so this implies that $J\cdot f-(J-\id)\cdot f=f$ lies in $I_b\cdot F+F^b$, contradicting our assumption.
\end{proof}

\para{Decomposing $F$ in terms of the subgroups $J_SF_{=S}$} We will also need, for a different purpose, a more specific version of \autoref{pr:bigb}. For each $S\in \Sigma(b)$, we have defined in \eqref{eq:defJS} the operator $J_S\in \Z[S_{2b}]$. For any $n\geq 2b$ we may consider this as an operator in $\Z[S_n]$, which we also denote by $J_S$. 
\begin{prop} 
\label{pr:indb}
For any $a\leq b$ such that $2b\leq n$,
\[F^{a,b+1}  \subset F^{a,b} + \sum_{S \in \Sigma(a,b)} J_S F_{=S}.\]
\end{prop}

\begin{proof}
For this proof only, define \begin{equation}
\label{eq:MpropJs}
F^{(a,b)}\coloneq F^{a,b}+\sum_{S\in\Sigma(a,b)}F_{=S}\quad\ \ =\ \big\langle f\colon [d]\into [n]\,\big\vert\,\nexists S\in \Sigma(a,b)\text{ s.t.\ } S\subsetneq \im f\cap [2b]\big\rangle.
\end{equation}
In words, $F^{(a,b)}$ is spanned by those injections $f\colon [d] \inj [n]$ such that $\im f \cap [2b]$ does not {\em properly} contain any element of $\Sigma(a,b)$ (but $\im f\cap[2b]$ is allowed to be {\em equal} to some $S\in\Sigma(a,b)$).

We begin by showing that $F^{a,b+1}\subset F^{(a,b)}$. Consider a basis element $f$ which does not lie in $F^{(a,b)}$. By definition there exists $S\in \Sigma(a,b)$ such that $S\subsetneq \im f\cap [2b]$. Choose $m\in \im f\cap [2b]$ with $m\notin S$, and define $T=S\cup\{m\}$. We have $T\in \Sigma(a,b+1)$ by \autoref{remark:Sigmab}\ref{part:Scupm}, so $f\notin F^{a,b+1}$ as desired.

We now show that for any $S\in \Sigma(a,b)$ we have: \begin{equation}
\label{eq:JsMnds}
F_{=S}\subset J_S F_{=S}+F^{a,b}+\sum_{\substack{S'\in \Sigma(a,b)\\S'\succ S}}F_{=S'}
\end{equation} Consider a basis element $f\in F_{=S}$ and the associated element \[(J_S-\id)\cdot f=\sum_{\sigma\neq 1\in (\Z/2)^S}(-1)^\sigma \sigma\cdot f.\]
By assumption $\im f\cap [2b]=S$, so $\im(\sigma\cdot f)\cap [2b]=\sigma\cdot \im f\cap [2b]=\sigma \cdot S$ is a descendant of $S$.
By \eqref{eq:lexfirst}, the fact that $S\in \Sigma(a,b)$ means that $\sigma\cdot S\succ S$ for all $\sigma\neq 1\in (\Z/2)^S$. Thus for each $\sigma$ there are two possibilities for the $b$-element subset $\sigma\cdot S$: either $\sigma\cdot S$ does not belong to $\Sigma(a,b)$, in which case $\sigma\cdot f\in F^{a,b}$; or $\sigma\cdot S\in \Sigma(a,b)$ but $\sigma\cdot S\succ S$, in which case $\sigma\cdot f\in F_{=\sigma\cdot S}$. In other words,  \[(\id-J_S)\cdot f\in F^{a,b}+\sum_{\substack{S'\in \Sigma(a,b)\\S'\succ S}}F_{=S'}.\]
Writing $f=J_S\cdot f-(J_S-\id)\cdot f$, this demonstrates \eqref{eq:JsMnds}.

Beginning with \eqref{eq:MpropJs}, we apply \eqref{eq:JsMnds} to each $S\in \Sigma(a,b)$ in lexicographic order to obtain the desired
\[F^{a,b+1}\quad\subset\quad
F^{(a,b)}=
F^{a,b}+\sum_{S\in\Sigma(a,b)}F_{=S}\quad\subset\quad
\sum_{S \in \Sigma(a,b)} J_S F_{=S}+F^{a,b}.\qedhere\]
\end{proof}

\subsection{Proof of  \texorpdfstring{\autoref{thm:saturationalternate}}{Theorem E'}}
\label{section:saturationproof}
We are now ready to apply the combinatorial apparatus above to FI-modules and prove \autoref{thm:saturationalternate}.
\begin{proof}[Proof of \autoref{thm:saturationalternate}]
We continue with the notation of \S\ref{section:combinatorics}, so  $F$ denotes  the $S_n$-module $\Z[\Hom_{\FI}([d],[n])]$, and $F^b$ and $F^{a,b}$ are the subgroups of $F$ defined in \autoref{def:Fsubgroups}.
Define subgroups $V^{b}\subset V_n$ and $V^{a,b}\subset V_n$ by $V^{b}\coloneq \im(F^{b}\otimes V_d\to V_n)$ and $V^{a,b}\coloneq \im(F^{a,b}\otimes V_d\to V_n)$. From the containments following \autoref{def:Fsubgroups} we see that $V^b=V^{1,b}\subset V^{2,b}\subset \cdots\subset V^{b,b}$.

Let us understand these subgroups $V^{a,b}$ more concretely. 
To say that $V$ is generated in degree $\leq d$ means that $V_n$ is spanned by its subgroups $V_T$ as $T$ ranges over subsets $T\subset [n]$ with $\abs{T}=d$. (Throughout this proof, $T$ will always denote a subset $T\subset [n]$ with $\abs{T}=d$.) 

By definition, $V^b$ is the subgroup of $V_n$ spanned by $f_*(V_d)$ where $f\colon [d]\into [n]$ ranges over injections for which $\im f$ does not contain any $S\in \Sigma(b)$.
In other words, \[V^b=\spn\big\{\,V_T\,\big\vert\,T\subset [n],\abs{T}=d\text{ s.t.\ $T$ does not contain any $S\in \Sigma(b)$}\,\big\}.\]
Similarly, $V^{b,b}$ is by definition the subgroup of $V_n$ spanned by those $V_T$ for which $[b]\not\subset T$ and $\abs{T}=d$. Since $V_{[n]-\{i\}}$ is the subgroup spanned by those $V_T$ where $i\notin T$, we see that \begin{equation}
\label{eq:Vbb}
V^{b,b}=V_{[n]-\{1\}}+\cdots+V_{[n]-\{b\}}.
\end{equation}

Fix some $n> K+d$ and some $a\leq n$. According to \eqref{eq:Vbb}, the desired conclusion of the theorem states that $\ker \J_{[a]}=V^{a,a}$ when $n>K+d$. From \eqref{eq:VMJ} we know that $V^{a,a}\subset \ker \J_{[a]}$ for \emph{all} $n$, so what we need to prove is that $\ker \J_{[a]}\subset V^{a,a}$ when $n>K+d$. We accomplish this by proving by reverse induction on $b$ that $\ker \J_{[a]}\subset V^{a,b}$ for all $b\geq a$.

Our base case is $b=K+1$. In this case we will prove something much stronger than the inductive hypothesis; we will prove  \autoref{cor:VKK} by showing that it is a direct consequence of \autoref{pr:bigb}. Recall that we always have the containments $V^{K+1}\subset V^{a,K+1}\subset V^{K+1,K+1}\subset V_n$.
The statement of \autoref{pr:bigb} for $b=K+1$ is that $F=I_{K+1}\cdot F+F^{\ell}$, and the hypothesis is satisfied since $n\geq (K+1)+d$. Therefore  \[V_n=\im(I_{K+1}\cdot F\otimes V_d)+V^{K+1}=I_{K+1}\cdot V_n+V^{K+1}.\]
Since $I_{K+1}\cdot V_n=0$ by assumption, we conclude that $V_n=V^{K+1}=V^{a,K+1}=V^{K+1,K+1}$. Notice that $V_n=V^{K+1,K+1}$ is precisely the conclusion of \autoref{cor:VKK}, as mentioned above. This concludes the base case.

For the inductive step, The key is to show that for all $a\leq b\leq K$ we have
\begin{equation}
\label{eq:Vab}
V^{a,b+1}\cap \ker\J_{[a]} \subset V^{a,b}.
\end{equation}
Given this, if we assume $\ker \J_{[a]}\subset V^{a,b+1}$ by induction, \eqref{eq:Vab} implies 
$\ker \J_{[a]}=V^{a,b+1}\cap \ker\J_{[a]} \subset V^{a,b}$,
which is the desired inductive hypothesis. The remainder of the argument thus consists of the proof of \eqref{eq:Vab}.

\medskip
For convenience, we would like to assume that $K\leq d$. If $K>d$, replacing $K$ by $d$ in the statement of \autoref{thm:saturationalternate} makes the conclusion stronger, while the hypothesis is still satisfied because $I_{d+1}\cdot V = 0$ by \autoref{pr:bigsat}. Therefore making this replacement if necessary, we may assume that $K\leq d$.
Our assumption on $n$ thus implies  $n>K+d\geq 2K\geq 2b$. Therefore we may apply \autoref{pr:indb}, which states that $F^{a,b+1}\subset F^{a,b}+\sum_{S\in \Sigma(a,b)}J_SF_{=S}$. We conclude that 
every $v\in V^{a,b+1}$ can be written as
\begin{equation}
\label{eq:vdecomp}
v=v^{a,b}+\sum_{S\in \Sigma(a,b)}v_S\qquad\qquad \text{where }v^{a,b}\in V^{a,b},\ v_S\in J_SF_{=S}\cdot V_d.
\end{equation}
It will suffice to show that if an element $v$ as in \eqref{eq:vdecomp} lies in $\ker \J_{[a]}$,
then in fact each term $v_S$ is zero, which implies \eqref{eq:Vab}.

Assume that $v\in V^{a,b+1}\cap \ker \J_{[a]}$, and suppose for a contradiction that $v_S \neq 0$ for some $S \in \Sigma(a,b)$. Let $S$ be the lexicographically first such element of $\Sigma(a,b)$. We may thus write
\begin{equation}
\label{eq:vdecompfirst}
v = v^{a,b} + v_S + \sum_{\substack{T\in \Sigma(a,b)\\ S\prec T}} v_T.
\end{equation}

For any $S\in \Sigma(a,b)$, write the elements of $S$ in order as $s_1<\cdots<s_b$, and define 
\[
\J_S \coloneq J_{s_1}^{\star_1}\cdots J_{s_b}^{\star_b}
\circ \incThmD\in \Z\big[\Hom_{\FI}([n],[n]\disjoint [\star_b])\big].\]
We will  establish a series of claims about $\J_S$, which hold for any $S\in \Sigma(a,b)$.
\begin{enumerate}[label={Claim~\arabic*.},ref={Claim~\arabic*},itemsep=2.5pt,parsep=2.5pt,topsep=4pt]
\item \label{item:JSkJa} $\J_S\cdot (\ker \J_{[a]})=0$.
\end{enumerate}

Proof of \ref{item:JSkJa}: To say that $S\in \Sigma(a,b)$ means that $[a]\subset S$, so the elements of $S$ are necessarily $1<2<\cdots<a<s_{a+1}<\cdots<s_b$. Therefore \[\J_S=J_1^{\star_1}\cdots J_a^{\star_a}J_{s_{a+1}}^{\star_{a+1}}\cdots J_{s_b}^{\star_b}=\J_{[a]}\cdot X\] Since $\J_S=\J_{[a]}\cdot X=X\cdot \J_{[a]}$, we have $\ker \J_{[a]}\subset \ker \J_S$ as claimed.
\proofofclaimend

By \eqref{eq:bigsatlemma}, we have 
\begin{equation}
\label{eq:JSfzero}
\J_S\cdot f=0\qquad\text{for any }f\colon [d]\into [n]\text{ with }S\not\subset \im f
\end{equation} since for any such $f$ there exists $s_i\notin \im f$, so $\{s_i,\star_i\}\cap \im f=\emptyset$. This has the following consequences.
\begin{enumerate}[resume*]
\item \label{item:JSFab} $\J_S\cdot F^{a,b}=0$ 
\item \label{item:JSJTFT} $\J_S\cdot J_TF_{=T}=0$ for any $T\in \Sigma(a,b)$ such that $S \prec T$.
\end{enumerate}

Proof of \ref{item:JSFab}: By definition any $f\in F^{a,b}$ has $S\not\subset\im f$, so $\J_S\cdot f=0$ by \eqref{eq:JSfzero}.\proofofclaimend

Proof of \ref{item:JSJTFT}: Given a generator $f\in F_{=T}$ we know that $\im f\cap [2b]=T$. As in the proof of \autoref{pr:indb}, the terms of $J_T\cdot f$ consist of $\sigma\cdot f$ for $\sigma\in (\Z/2)^T$. The intersections $\im(\sigma\cdot f)\cap [2b]$ are precisely the descendants $\sigma\cdot T$. Since $T\in \Sigma(a,b)$ we have $\sigma\cdot T\succcurlyeq T$. In particular, since $S\prec T\preccurlyeq \sigma\cdot T$, every term satisfies $S\not\subset \im(\sigma\cdot f)$. By \eqref{eq:JSfzero}, $\J_S\cdot J_TF_{=T}=0$ as desired.\proofofclaimend

\medskip
We now apply these consequences to the decomposition \eqref{eq:vdecompfirst}. By \ref{item:JSkJa}, our assumption that $v\in \ker \J_{[a]}$ implies that $\J_S\cdot v=0$. 
\ref{item:JSFab} and \ref{item:JSJTFT} show that $\J_S\cdot v^{a,b}=0$ and $\J_S\cdot v_T=0$. We conclude that $\J_S\cdot v_S=\J_S\cdot v=0$; it remains to show that this implies $v_S=0\in V_n$.

We show this using the following two claims, which we prove in turn. Define $\tau\in \End_{\FI}([n]\disjoint [\star_b])$ to be the involution $\tau\coloneq (t_1\ \,\star_1)\cdots(t_b\ \,\star_b)$, where $(t_1, \ldots, t_b)$ denotes the complement of $S$ in $[2b]$ as before.

\begin{enumerate}[resume*] 
\item \label{item:JSJSFS} $\J_S\cdot J_S=\J_S$ when restricted to $F_{=S}$. 
\item \label{item:tauJS} $\tau\J_s=\incThmD \circ J_S$ when restricted to $F_{=S}$.
\end{enumerate}

Proof of \ref{item:JSJSFS}: As in \ref{item:JSJTFT}, given $f\in F_{=S}$ with $\im f\cap [2b]=S$, the terms of $J_S\cdot f$ consist of $f$ together with $\sigma\cdot f$ for $\sigma\neq 1\in (\Z/2)^S$. Each of the latter has $\im(\sigma\cdot f)\cap [2b]=\sigma\cdot S\succ S$. Therefore $S\not\subset \sigma\cdot f$ for $\sigma\neq 1$, so $\J_S\cdot \sigma\cdot f=0$ by \eqref{eq:JSfzero}. We conclude that $\J_S\cdot J_S\cdot f=\J_S\cdot (f+\sum(-1)^\sigma \sigma\cdot f)=\J_S\cdot f$, as claimed.\proofofclaimend

Proof of \ref{item:tauJS}: Note that $\tau(J_{s_1}^{\star_1}\cdots J_{s_b}^{\star_b}
)\tau^{-1}=J_{s_1}^{t_1}\cdots J_{s_b}^{t_b}$. Therefore $\tau\J_S=J_{s_1}^{t_1}\cdots J_{s_b}^{t_b}\circ \tau\circ \incThmD\in \Z\Hom_{\FI}([n],[n]\disjoint [\star_b])$. By definition, the image of  a map $f\in F_{=S}$ does not contain $t_i$, so when restricted to $F_{=S}$ we have $\tau\circ \incThmD=\incThmD$. We conclude that $\tau\J_S=J_{s_1}^{t_1}\cdots J_{s_b}^{t_b}\circ \incThmD=\incThmD\circ J_S$, as claimed.\proofofclaimend

We now complete the proof. Write $v_S=J_S\cdot w_S$ for $w_S\in F_{=S}\cdot V_d\subset V_n$.  \ref{item:JSJSFS} implies that $\J_S\cdot v_S=\J_S\cdot J_S\cdot w_S=\J_S\cdot w_S$. Thus $\J_S\cdot w_S=0$, so certainly $\tau\J_S w_S=0$. \ref{item:tauJS} implies that $\tau \J_S w_S=\incThmD(J_S w_S)=\incThmD(v_S)$. Combining these, we see that $\incThmD(v_S)=0$. Since 
$V$ is torsion-free, $\incThmD$ is injective, so this proves that $v_S=0$.

This contradicts our assumption that $v_S\neq 0$, so we conclude from \eqref{eq:vdecomp} that $v\in V^{a,b}$. This concludes the proof of the containment \eqref{eq:Vab}; as we explained following \eqref{eq:Vab}, this completes the proof of the inductive hypothesis and thus concludes the proof of the theorem.
\end{proof}

\section{Bounds on the homology  of FI-modules}
\label{section:bounds-and-proof-of-thm-A}

\para{An outline of the proof of \autoref{main:regularity}}
Before launching into the proof of \autoref{main:regularity}, we outline the steps that we will take. Recall that \autoref{main:regularity} states that for an FI-module $W$, the degree of the FI-homology $H_p(W)$ can be bounded in terms of certain invariants of $W$. In this outline, whenever we speak of a ``bound on'' a particular FI-module, we mean a bound on its degree.
\begin{enumerate}
\item \label{step:DaH} We prove that a bound on $D^a X$ can be converted to a bound on $H_0(X)$ (\autoref{pr:da}).
\item We show that \autoref{main:saturation} gives a bound on the degree of $H_1^{D^a}(W)$ for all $a$.
\item Using homological properties of the functor $D$, we show that this bound on $H_1^{D^a}(W)$ implies a bound on $H_p^{D^a}(W)$ for all $p$ and all $a$.
\item If $X_p$ is the $p$-th syzygy of $W$,\footnote{Here we consider syzygies relative to a free resolution of $W$ that is \emph{minimal} in the sense that all maps become 0 after applying $H_0$.} it is almost true that $H_p(W)=H_0(X_p)$; specifically, we have $H_p(W)=H_1(X_{p-1})$ and $H_p(W)\into H_0(X_p)$. Similarly, it is almost true that $H_p^{D^a}(W)=D^a(X_p)$, and we prove that for sufficiently large $a$ this \emph{is} true. Therefore by using Step~\ref{step:DaH}, we can convert our bound on $H_p^{D^a}(W)$ to the desired bound on $H_p(W)$.
\end{enumerate}

\subsection{Relations and \texorpdfstring{$H_1$}{H-1}}
\label{sec:relH}
Our main theorems will be proved in terms of a presentation of the FI-module in question. We saw in \autoref{lem:genequiv} that $W$ is generated in degree $\leq k$ if and only if $\deg H_0(V)\leq k$. The existence of a presentation for $W$ with relations in degree $\leq d$ is very close to the condition $\deg H_1(W)\leq d$, but they are not quite equivalent.\footnote{For instance, a FI-module $W$ admitting a finite-length filtration whose graded pieces are free has $H_1(W) = 0$, but such a $W$ need not itself be free (recall that free FI-modules need not be projective).  If we could always find a surjection $M(H_0(W))\onto W$ lifting the isomorphism on $H_0$ there would be no problem, but such a surjection does not always exist. For example, it can happen that $H_0(W)_n\iso \Z/2\Z$ while $W_n$ is a free abelian group, in which case there is no map $H_0(W)_n\to W_n$ at all.} Therefore we distinguish these in our terminology as follows.
\begin{definition}
\label{def:genrel}
We say that an FI-module $W$ is generated in degree $\leq k$ and related in degree $\leq d$ if there exists a short exact sequence \[0\to V\to M\to W\to 0\] where $M$ is a free FI-module generated in degree $\leq k$ and   $V$ is generated in degree $\leq d$.
\end{definition}

\begin{proposition}
\label{pr:genrel}
Any FI-module $W$ is generated in degree $\leq \deg H_0(W)$ and related in degree $\leq \max(\deg H_0(W),\deg H_1(W))$. 
\end{proposition}
\begin{proof}
Set $M\coloneq M(W_{\langle \leq \deg H_0(W)\rangle})$. By \autoref{lem:genequiv}, the natural map $M\onto W$ is surjective.

Let $V$ be its kernel, so that $0\to V\to M\to W\to 0$ is a presentation of $W$ as in \autoref{def:genrel}. By \autoref{lemma:freeHacyclic} $M$ is $H_0$-acyclic, so we have the exact sequence:
\begin{displaymath}
0 \ra H_1(W) \ra H_0(V) \ra H_0(M)
\end{displaymath}
From this we conclude $\deg H_0(V)$ is bounded by the degrees of the other two terms. Since $H_0(M)=W_{\leq \deg H_0(W)}$, we see in particular that  $\deg H_0(M)=\deg H_0(W)$. Therefore $V$ is generated in degree $\leq \max(\deg H_0(W),\deg H_1(W))$, as desired.
\end{proof}

From this proposition we see that relations will indeed behave as we would expect, as long as $\deg H_0(W)\leq \deg H_1(W)$. We will reduce to this case in the proof of \autoref{main:regularity} using the following proposition, whose proof was explained to us by Eric Ramos; we are grateful to the referee for suggesting the current statement. Similar arguments appear in Li--Yu \cite{LiYu} in the proof of Corollary 3.4 and the second proof of Lemma 3.3.

Given an FI-module $W$ and some $m\geq 0$, consider the FI-module $Z=W_{\langle \leq m\rangle}/W_{\langle <m\rangle}$. Note that $H_0(Z)$ vanishes except in degree $m$, where $H_0(Z)_m=Z_m=H_0(W)_m$, so we have a surjection $M(Z_m)\onto Z$. In terms of the original FI-module $W$, we have a natural surjection from $M(H_0(W)_m)$ to $W_{\langle \leq m\rangle}/W_{\langle <m\rangle}$ which is an isomorphism in degree $m$.
\begin{proposition}
\label{pr:ramosCE}  Let $W$ be an FI-module with $\deg H_0(W)<\infty$. The natural surjection  
\[M(H_0(W)_m)\onto W_{\langle \leq m\rangle}/W_{\langle <m\rangle}\] is an isomorphism whenever $m\geq \deg H_1(W)$ or $m > \deg H_0(W)$. 
In particular, 
the inclusion $W_{\langle <\deg H_1(W)\rangle }\into W$ induces an isomorphism on $H_i$ for all $i>0$.
\end{proposition}
\begin{proof}
We proceed by reverse induction on $m$, showing both that $M(H_0(W)_m)\iso W_{\langle \leq m\rangle}/W_{\langle <m\rangle}$ \emph{and} that the inclusion $W_{\langle <m\rangle}\into W$ induces isomorphisms on $H_i$ for all $i>0$. 
Our base case consists of all $m>\deg H_0(W)$, when both claims are essentially tautological: in this case $M(H_0(W))_m=0$ and $W_{\langle <m\rangle}= W_{\langle \leq m\rangle}=W$, so both sides of the claimed isomorphism vanish, proving the first claim. Similarly $W_{\langle <m \rangle}=W$ if $m>\deg H_0(W)$, so the second claim is automatic.

For ``usual'' FI-modules with $\deg H_0(W)<\deg H_1(W)$ there is nothing left to prove; it remains to handle FI-modules with $\deg H_1(W)\leq \deg H_0(W)$. 

For the inductive step, 
write $Z$ for the quotient $Z\coloneq W_{\langle \leq m\rangle}/W_{\langle <m\rangle}$, and let $A$ be the kernel of the surjection $M(Z_m)\onto Z$, so that $0\to A\to M(Z_m)\to Z\to 0$. The FI-module $A$ vanishes in degrees $\leq m$: in degree $m$ the map $M(Z_m)_m\to Z_m$ is an isomorphism, and $M(Z_m)$ itself is zero in degrees $<m$. Since $A$ vanishes in degrees $\leq m$, $H_0(A)$ also vanishes in degrees $\leq m$; since $H_0(M(Z_m))$ vanishes in degrees $>m$, the map $H_0(A)\to H_0(M(Z_m))$ is zero. Since $M(Z_m)$ is $H_0$-acyclic, we conclude that there is an isomorphism $H_1(Z)\iso H_0(A)$.

Now consider the long exact sequence 
\[\cdots\to H_1(W_{\langle \leq m\rangle})\to H_1(Z)\to H_0(W_{\langle <m\rangle})\to\cdots.\]
By induction, we know that $\deg H_1(W_{\langle \leq m\rangle})=\deg H_1(W)\leq m$, and by definition $\deg H_0(W_{\langle <m\rangle})<m$. Therefore $\deg H_1(Z)\leq m$.
 
We showed above that $H_1(Z)$ vanishes in degrees $>m$, while $H_0(A)$ vanishes in degrees $\leq m$, so $H_1(Z)=H_0(A)=0$. Therefore $A=0$, and the natural map $M(H_0(W)_m)=M(Z_m)\to Z$ is an isomorphism, as claimed.
Since free FI-modules are $H_0$-acyclic, we conclude that the inclusion $W_{\langle <m\rangle}\into W_{\langle \leq m\rangle}$ induces an isomorphism on $H_i$ for all $i>0$; the inclusion $W_{\langle \leq m\rangle}\into W$ induces an isomorphism on $H_i$ for all $i>0$ by induction, so we have proved the inductive hypothesis.
\end{proof}

\subsection{Homological properties of the derivative}
Considering $\FBMod$ as a full subcategory of $\FIMod$, the functor $S$ restricts to a functor $S\colon \FBMod\to \FBMod$. 
\begin{lemma}
\label{lemma:DMMS}
There is a natural isomorphism of functors $D\circ M=M\circ S\colon \FBMod\to \FIMod$.
\end{lemma}
\begin{proof} There is automatically a natural transformation $M\circ S\to D\circ M$. It would suffice to check that this is an isomorphism on free FB-modules, but it will be no more difficult to check this on arbitrary FB-modules $W$. From the formula \eqref{eq:MW} for $M(W)$ we see that $(SM(W))_T=\bigoplus_{S\subset T\disjoint\{\star\}}W_S$, with $\iSV\colon M(W)\to SM(W)$ the inclusion of those summands with $\star\not\in S$. (Incidentally, this shows that free FI-modules are torsion-free.) It follows that
\[(DM(W))_T=\bigoplus_{\substack{S\subset T\disjoint\{\star\}\\ \star\in S}}W_S=\bigoplus_{R\subset T}W_{R\disjoint \{\star\}}=\bigoplus_{R\subset T}(SW)_R=M(SW)_T,\] as claimed (for the second equality we reindex by $S=R\disjoint \{\star\}$). It is straightforward to check that this identification agrees on morphisms as well.
\end{proof}

\begin{corollary}
\label{cor:freeacyclic}
Free FI-modules are $D^a$-acyclic for all $a\geq 1$.\end{corollary}
\begin{proof}
Just as in the proof of \autoref{lemma:freeHacyclic}, we have $H_p^{D^a}(M(W))\iso H_p^{D^a\circ M}(W)$. However \autoref{lemma:DMMS} implies that $D^a\circ M=M\circ S^a$. This is exact since both $M$ and $S$ are, so $H_p^{M\circ S^a}=0$ for all $p>0$.
\end{proof}

\begin{prop}
\label{pr:da}
If $V$ is an FI-module generated in degree $\leq k$, then $D^a V = 0$ for all $a > k$.  On the other hand,
if $\deg D^a V \leq m$ for some $m\geq -1$, then $V$ is generated in degree $\leq m+a$.
\end{prop}

\begin{proof}If $V$ is generated in degree $\leq k$, there is a surjection $M(V_{\leq k})\onto V$. Since $D^a$ is right-exact, we have a surjection $D^a M(V_{\leq k})\onto D^a V$ for any $a$. By \autoref{lemma:DMMS}, $D^a M(V_{\leq k})\iso M(S^a V_{\leq k})$. However, $S^aV_{\leq k}=0$ when $a>k$, since $(S^aV_{\leq k})_R=(V_{\leq k})_{R\disjoint [\star_a]}=0$. Therefore $D^a V=0$ for $a>k$.

For the second claim, to say that $\deg D^a V\leq m$ means that $(D^a V)_T=0$ whenever $\abs{T}>m$. The formula \eqref{eq:dkdisc} for $(D^aV)_T$ shows that the defining surjection $V_{T\disjoint [\star_a]}\onto H_0(V)_{T\disjoint [\star_a]}$ factors through $(D^aV)_T\onto H_0(V)_{T\disjoint [\star_a]}$, so it follows that $H_0(V)_R=0$ whenever $\abs{R}>m+a$. In other words, $V$ is generated in degree $\leq m+a$.
\end{proof}

\para{The derived functors of $D$}
We can now establish the basic properties of the derived functors $H_p^D$ of the derivative $D$.
\pagebreak

\begin{lemma}Let $W$ be an FI-module.
\label{lemma:Dprops}
\begin{enumerate}[label={\normalfont (\roman*)}]
\item \label{item:H1Dseq} The derived functor $H_1^D$ coincides with $K$, so there is a natural exact sequence \[0\to H_1^D(W)\to W\to SW\to DW\to 0.\]
\item \label{item:torsionfree} $W$ is torsion-free if and only if $H_1^D(W)=0$.
\item \label{item:Ddim1} $H_p^D=0$ for all $p>1$.
\item \label{item:Dprojective} $D$ takes projective FI-modules to projective FI-modules. 
\item \label{item:Ddegree} If $Y$ is an FI-module of finite degree, then $\deg DY\leq \deg Y-1$ and $\deg H_1^D(Y)\leq \deg Y$.
\end{enumerate}
\end{lemma}
\begin{proof} Given $W$, let $M$ be a free FI-module with $M\onto W$; for instance, we may take the universal $M=M(W)\onto W$. Let $V$ be the kernel of this surjection, so we have $0\to V\to M\to W\to 0$. Since $M$ is free, $H_1^D(M)=0$ by \autoref{cor:freeacyclic}, so we have an isomorphism $H_1^D(W)\iso \ker(DV\to DM)$. 

\ref{item:H1Dseq} The key properties are that $S$ is exact and that $M\xrightarrow{\iSV} SM$ is injective, i.e.\ that free FI-modules are torsion-free (which we saw in the proof of \autoref{lemma:DMMS}). 
Thus we have a diagram
\[\xymatrix{
 & V \ar[r] \ar[d] & SV \ar[r] \ar[d] & DV \ar[r] \ar[d] & 0 \\
0 \ar[r] & M \ar[r]  & SM \ar[r]  & DM \ar[r]  & 0 }\]
Applying the snake lemma, we obtain the desired exact sequence
\[\ker(SV\to SM)=0\to H_1^D(W)\to W\to SW\to DW\to 0.\]
In particular, this identifies $H_1^D(W)$ with $KW=\ker(W\to SW)$.

\ref{item:torsionfree} Given that $H_1^D=K$, this is the statement of \autoref{lem:torsion-freeK}.

\ref{item:Ddim1} Since $M$ is $D$-acyclic, we have $H_2^D(W)\iso H_1^D(V)$. The FI-module $V$ is torsion-free, being a submodule of $M$, so $H_1^D(V)=0$ by \ref{item:torsionfree}. Since $W$ was arbitrary, this proves that $H_2^D=0$, which implies that $H_p^D=0$ for all $p>1$.

\ref{item:Dprojective} Since projective FI-modules are summands of $\bigoplus M(m_i)$, it suffices to prove this for $M(m)=M(\Z[S_m])$. \autoref{lemma:DMMS} states that \[DM(m)=DM(\Z[S_m])\iso M(S\Z[S_m])\iso M(\bigoplus_{i=1}^m \Z[S_{m-1}])=\bigoplus_{i=1}^m M(m-1)\] which is indeed projective.

\ref{item:Ddegree} It is clear that $\deg SY=\deg Y-1$, since $(SY)_n=Y_{[n]\disjoint[\star]}\iso Y_{n+1}$ (unless $\deg Y=0$, when $\deg SY=-\infty$). Both claims now follow from \ref{item:H1Dseq}, the first from the surjection $SY\onto DY$ and  the second  from the injection $H_1^D(Y)\into Y$.
\end{proof}

\subsection{Proof of \autoref{main:regularity}}
\label{sec:thmAproof}
We now have in place all the tools we need to prove our main theorems bounding the degree of homology of FI-modules.
The key technical result is \autoref{main:saturation}, together with \autoref{lem:DaVM} and \autoref{rem:saturationreinterpreted} establishing a connection between its conclusion and $D^a$.
\begin{theorem}
\label{thm:Dkhomology}
Let $W$ be an FI-module generated in degree $\leq k$ and related in degree $\leq d$, and let $N \coloneq d+\min(k,d)-1$.  For all $a\geq 1$ and all $p\geq 1$, 
\begin{equation}
\tag{$\ast^a_p$}
\label{eq:ap}
\deg H_p^{D^a}(W) \leq N - a + p.
\end{equation}
\end{theorem}
\begin{proof}
We will reduce by induction to the case when $a=1$ or $p=1$. To accomplish this reduction, we prove that $(\ast^{a-1}_p) + (\ast^{a-1}_{p-1})\implies (\ast^a_p)$ for any $a\geq 2$ and $p\geq 2$.

Fix $a\geq 2$ and $p\geq 2$.
By  \autoref{lemma:Dprops}\ref{item:Dprojective},  $D^a$ takes projective FI-modules to projective FI-modules, so we may compute the left derived functors of $D^a$ by means of the Grothendieck spectral sequence applied to the composition $D \circ D^{a-1}$. Thanks to the vanishing of $H_p^D$ for $p > 1$ from  \autoref{lemma:Dprops}\ref{item:Ddim1}, this spectral sequence has only two nonzero columns, so it degenerates to the short exact sequences
\begin{equation}
\label{eq:spectral}
0\to D H_p^{D^{a-1}}(W) \ra H_p^{D^a}(W) \ra  H_1^D(H_{p-1}^{D^{a-1}} W)\to 0.
\end{equation}

The assertions $(\ast^{a-1}_p)$ and $(\ast^{a-1}_{p-1})$ state respectively that
\begin{align*}
\deg H_p^{D^{a-1}}(W) &\leq N- (a-1) + p\phantom{(p-1)} =N-a+p+1\\
\deg H_{p-1}^{D^{a-1}}(W) &\leq N- (a-1) + (p-1)\phantom{p} =N-a+p
\end{align*}
\autoref{lemma:Dprops}\ref{item:Ddegree} tells us that $\deg DY\leq \deg Y-1$ and $\deg H_1^D(Y)\leq \deg Y$, so these bounds imply:
\begin{align*}
\deg D H_p^{D^{a-1}}(W) 
&\leq N-a+p\\
\deg H_1^D(H_{p-1}^{D^{a-1}}(W)) 
&\leq N-a+p
\end{align*}
The short exact sequence \eqref{eq:spectral} now implies
\[
\deg H_p^{D^a}(W) \leq N - a + p
\]
which is precisely the assertion $(\ast^a_p)$. 
This concludes the proof that $(\ast^{a-1}_p) + (\ast^{a-1}_{p-1})\implies (\ast^a_p)$ for any $a\geq 2$ and $p\geq 2$.

Given this implication, it suffices to prove directly that $(\ast^a_p)$ holds when either $a=1$ or $p=1$, since all remaining cases with $a\geq 2$ and $p\geq 2$ then follow by induction. When $a=1$ and $p\geq 2$, we have $H_p^D(W)=0$ by \autoref{lemma:Dprops}\ref{item:Ddim1}, so $\deg H_p^D(W)=-\infty$ and the bound $(\ast^a_p)$ certainly holds. What remains as the unavoidable core of the problem is  the bound \eqref{eq:ap} when $p=1$, namely that $\deg H_1^{D^a}(W)\leq N-a+1$ for all $a\geq 1$.

To compute $H_1^{D^a}(W)$, consider a presentation $0 \ra V \ra M \ra W \ra 0$ as in \autoref{def:genrel}, with $M$ free and generated in degree $\leq k$ and $V$ generated in degree $\leq d$.
 Since $M$ is free, it is $D^a$-acyclic by \autoref{cor:freeacyclic}, so
\begin{displaymath}
H_1^{D^a}(W) \iso \ker\big(D^a V\to D^a M\big).
\end{displaymath}
But recall from \autoref{rem:saturationreinterpreted} that the conclusion of \autoref{main:saturation} can be restated as a claim about the map $D^a V\to D^aM$ and its kernel! Specifically,  the conclusion of \autoref{main:saturation} says for any $a\geq 1$ that \[\quad\quad\ker(D^a V\to D^a M)_{n-a} =0\quad\text{ for all $n>d+\min(k,d)$},\] or in other words that
\[\deg \ker(D^a V\to D^a M)\leq d+\min(k,d)-a=N+1-a.\] This means that the conclusion of \autoref{main:saturation} applied to the relations $V\subset M$ is precisely the claim \eqref{eq:ap} for $p=1$ and all $a\geq 1$. As explained above, all other cases now follow by induction.
\end{proof}

\begin{proof}[Proof of \autoref{main:regularity}]
Fix $k'\geq 0$ and $d\geq 0$, and let $U$ be an FI-module with $\deg H_0(U)\leq k'$ and $\deg H_1(U)\leq d$. Our goal is to prove that $\deg H_p(U)-p\leq k'+d-1$ for all $p>0$.

We first reduce to the case when $k'<d$. Let $k\coloneq\min(k',d-1)$ and define $W$ to be the submodule $W\coloneq U_{\langle \leq k\rangle}$. In the most common case when $k'<d$, this has no effect: we have $k=k'$ and $W=U$. In the other case when $k'\geq d$, we have $\deg H_1(U)\leq d=k+1$, so \autoref{pr:ramosCE} states that $H_p(W)\iso H_p(U)$ for all $p>0$. Since $k\leq k'$ in either case, to prove the theorem it suffices to prove that \[\deg H_p(W)-p\leq k+d-1\qquad\text{ for all $p>0$.}\] For the rest of the proof, we discard the FI-module $U$ and work only with $W$, which has $\deg H_0(W)\leq k$ and $\deg H_1(W)\leq d$ with $k<d$.

Given these bounds,
\autoref{pr:genrel} tells us that $W$ is generated in degree $\leq k$ and related in degree $\leq \max(k,d)=d$.
Therefore there exists a surjection $M\onto W$ from a free FI-module $M$ generated in degree $\leq k$, whose kernel is generated in degree $\leq d$. Set $M_0\coloneq M$ and extend this to a resolution of $W$ by free FI-modules
\begin{displaymath}
\ldots \ra M_2 \ra M_1 \ra M_0 \ra W \ra 0.
\end{displaymath}

For each $p > 0$, let $X_p$ be the $p$th syzygy of $W$, namely $X_p\coloneq \im(M_p\to M_{p-1})\iso\ker(M_{p-1}\to M_{p-2})$. Let us assume that this resolution is {\em minimal} in the very weak sense that $\deg H_0(X_p)=\deg H_0(M_p)$ for all $p> 0$. 
(The existence of such a resolution is a consequence of the fact that every FI-module $V$ generated in degree $\leq k$ admits a surjection from a free FI-module generated in degree $\leq k$, namely $M(V_{\leq k})$ as discussed in \autoref{pr:da}.)  Set $X_0\coloneq W$. 

For all $p\geq 1$ we have an exact sequence
\begin{equation}
\label{eq:syzygy}
0 \ra X_{p} \ra M_{p-1} \ra X_{p-1} \ra 0.
\end{equation} Since the $M_i$ are $H_0$-acyclic by \autoref{cor:freeacyclic}, applying $H_0$ to \eqref{eq:syzygy} gives $H_i(X_p)\iso H_{i+1}(X_{p-1})$ for all $i\geq 1$; iterating, we obtain $H_p(W)\iso H_1(X_{p-1})$. Similarly, $H_p^{D^a}(W)\iso H_1^{D^a}(X_{p-1})$ for any $a\geq 1$.

Let us write \[N\coloneq d+k-1,\] so our eventual goal is to prove that $\deg H_p(W)\leq N+p$ for all $p\geq 1$.
Before that, we  prove that for all $p\geq 1$,
\begin{equation}
\label{eq:ih}
\deg H_0(X_p) \leq N+p.
\end{equation}

By construction $X_1=\ker(M\to W)$; by our hypothesis, $X_1$ is generated in degree $\leq d$,  so $\deg H_0(X_1)\leq d\leq d+k=N+1$. This proves \eqref{eq:ih} for $p=1$; we proceed by induction on $p$.

Fix $p\geq 2$, and assume by induction that \eqref{eq:ih} holds for $p-1$, i.e.\ that $\deg H_0(X_{p-1}) \leq N+p-1$. By minimality of the resolution, $\deg H_0(M_{p-1}) = \deg H_0(X_{p-1})$, so $M_{p-1}$ is generated in degree $\leq N+p-1$. By \autoref{pr:da}, this implies that $D^{N+p} M_{p-1} = 0$. Then applying $D^{N+p}$ to \eqref{eq:syzygy} yields a long exact sequence containing the segment:
\begin{displaymath}
H_1^{D^{N+p}}(X_{p-1}) \ra D^{N+p} X_p \ra  0=D^{N+p} M_{p-1}
\end{displaymath}

This shows that $D^{N+p} X_p$ is a quotient of $H_1^{D^{N+p}}(X_{p-1})\iso H_p^{D^{N+p}}(W)$. We proved in \autoref{thm:Dkhomology} that
\begin{displaymath}
\deg H_p^{D^{N+p}}(W) \leq N - (N+p) + p = 0,
\end{displaymath}
so $\deg D^{N+p} X_p \leq 0$. (The statement of \autoref{thm:Dkhomology} has $d+\min(k,d)-1$, but this coincides with $N=d+k-1$ since $k<d$.)  By \autoref{pr:da}, this implies that $X_p$ is generated in degree at most $N + p$, which is the result to be proved. This concludes the proof of \eqref{eq:ih}.

We saw above that \eqref{eq:syzygy} implies $H_i(X_p)\iso H_{i+1}(X_{p-1})$ for $i\geq 1$. To complete the proof of the theorem, we consider the segment of the long exact sequence involving $i=0$:
\begin{displaymath}
0 = H_1(M_{p-1}) \ra H_1(X_{p-1}) \ra H_0(X_p) \ra H_0(M_{p-1}) \ra H_0(X_{p-1})
\end{displaymath}
This shows that $H_p(W)\iso H_1(X_{p-1})$ injects into $H_0(X_p)$ for all $p>0$. We proved in \eqref{eq:ih} that $\deg H_0(X_p)\leq N+p$, so we conclude that $\deg H_p(W)\leq N+p$ for all $p>0$, as desired.
\end{proof}

\section{Application to homology of congruence subgroups}

\subsection{A complex computing \texorpdfstring{$H_i(V)$}{H-i(V)}}

For any category $\C$, let $\C\dMod$ denote the category of functors $\C\to \ZMod$. Given $V\in \C\dMod$ and $W\in \C^{\op}\dMod$, their tensor product over $\C$ is an abelian group $V\otimes_\C W$. It can be defined as the largest quotient of \[\bigoplus_{X\in \Ob \C} V(X)\otimes_\Z W(X)\] in which \[v_X\otimes f^*(w_Y)\in V(X)\otimes W(X)\quad\text{ is identified with }\quad f_*(v_X)\otimes w_Y\in V(Y)\otimes W(Y)\] for all $X,Y\in \Ob\C$, $v_X\in V(X)$, $w_Y\in W(Y)$, and $f\in\Hom_{\C}(X,Y)$.\footnote{The reader may recognize this as an example of a \emph{coend}: given $V$ and $W$ we can define a functor $V\boxtimes W\colon \C\times \C^{\op}\to \ZMod$; then $V\otimes_\C W$ coincides with the coend $\int^C V\boxtimes W$, and the quotient construction above is just the standard coequalizer formula for a coend.}

In this paper we will be interested in the tensor product of an FI-module $V$ and co-FI-module $W$. This can be described explicitly as follows.
\begin{definition}
\label{de:fitensor}
Given $V\in \FIMod$ and $W\in \FI^{\op}\dMod$, the abelian group $V\otimes_{\FI} W$ is defined by:
\begin{align*}
V\otimes_{\FI} W&=
 \big(\bigoplus_{T\in \Ob \FI} V_T\otimes_\Z W_T\big)/\big\langle f_*(v_S)\otimes w_T\equiv v_S\otimes f^*(w_T) \,\big|\,f\colon S\into T\big\rangle\\
 &=\big(\bigoplus_{n\geq 0} V_n\otimes_{\Z S_n} W_n\big)/\big\langle f_*(v_n)\otimes w_{n+1}\equiv v_n\otimes f^*(w_{n+1}) \,\big|\,f\colon [n]\into [n+1]\big\rangle
\end{align*}
\end{definition}

We think of an FI-module $V\in \FIMod$ as a ``right module over FI'', and a co-FI-module $W\in \FIopMod$ as a ``left module over FI''. This is consistent with our notation $V\otimes_{\FI}W$ for the tensor. Moreover, if $W$ is an $\FI^{\op}\times \FI$-module, we will say that $W$ is an \emph{$\FI$-bimodule}; in this case $V\otimes_{\FI}W$ is not just an abelian group, but in fact an FI-module. This is familiar from the analogous situation with $R$-modules: the tensor of a right $R$-module with an $R$-bimodule is a right $R$-module. To verify the claim in this setting, just note that \[(\FI^{\op}\times \FI)\dMod=[\FI^{\op}\times \FI,\ZMod]=[\FI,[\FI^{\op},\ZMod]]=[\FI,\FIopMod].\] In other words, we can think of an FI-bimodule $W$ as a functor from $\FI$ to $\FIopMod$; after tensoring with $V\in \FIMod$, we are left with a functor from $\FI$ to $\ZMod$, which is just an FI-module.

\begin{definition}
The $\FI^{\op}\times \FI$-module $K$ is defined on objects by $K(S,T)=\Z[\Bij(S,T)]$; in particular $K(S,T)=0$ if $\abs{S}\neq \abs{T}$. Given a morphism \[(f\colon S'\into S, g\colon T\into T')\text{\ \ in\ \ } \Hom_{\FI^{\op}\times \FI}\big((S,T),(S',T')\big),\] we consider two cases. If $f$ and $g$ are both bijective,  $K_{(f,g)}\colon K(S,T)\to K(S',T')$ is the map defined by $\Bij(S,T)\ni \varphi \mapsto g\circ \varphi \circ f\in \Bij(S',T')$. If either $f$ or $g$ is not bijective,  $K_{(f,g)}=0$.
\end{definition}

Since $K$ is an FI-bimodule, the tensor $V\otimes_{\FI} K$ is itself an FI-module. In fact, this FI-module is already familiar to us! To avoid confusion, in the remainder of the paper we will write $H_i^{\FI}(V)$ for the FI-homology of $V$, which was denoted simply by $H_i(V)$ in previous sections.
\begin{proposition}
\label{prop:TorK}
Given $V\in \FIMod$, the FI-module $V\otimes_{\FI} K$ is isomorphic to the FI-module $H^{\FI}_0(V)$ defined in the introduction. As a consequence, 
\[H^{\FI}_i(V)=\Tor^{\FI}_i(V,K)\quad \text{for any }i\geq 0.\]
\end{proposition}
\begin{proof}
\autoref{de:fitensor} presents $V\otimes_{\FI} K$ as a quotient of
\begin{displaymath}
\bigoplus_{n\geq 0} V_n\otimes_{\Z S_n} K_n,
\end{displaymath}
so we first identify the FI-module $V_n\otimes_{\Z S_n} K_n$. Since $K$ is not only a co-FI module but an FI-bimodule, $K_n$ is an $S_n\times \FI$-module: as an FI-module $K_n$ sends a set $T$ to $\Z[\Bij([n],T)]$, and the action of $S_n$ by precomposition commutes with this FI-module structure. Thus the FI-module 
$V_n\otimes_{\Z S_n} K_n$
sends $T$ to $V_T$ if $\abs{T}=n$, and to $0$ if $\abs{T}\neq n$. Passing to the direct sum, we find that $\bigoplus_{n\geq 0} V_n\otimes_{\Z S_n} K_n$ sends $T$ to $V_T$ for any finite set $T$ of any cardinality; in other words, the FI-module $\bigoplus_{n\geq 0} V_n\otimes_{\Z S_n} K_n$ can be identified with $V$ itself.

We now consider the relations: \autoref{de:fitensor} states that $V\otimes_{\FI} K$ is the quotient of
$V\simeq \bigoplus V_n\otimes_{\Z S_n} K_n$ by the relations
\[f_*(v_n)\otimes k_{n+1}\equiv v_n\otimes f^*(k_{n+1})\qquad \text{for all }f\colon [n]\into [n+1].\]
However, by definition $f^*$ acts as 0 on $K$ whenever $f$ is not bijective. Therefore these relations reduce to $f_*(v_n)\equiv 0$ for all $v_n\in V_n$ and $f\colon [n] \inj [n+1]$.  The quotient of $\bigoplus_n V_n$ by these relations is precisely $H^{\FI}_0(V)$ as we defined it in the introduction. The assertion that $H^{\FI}_i(V)=\Tor^{\FI}_i(V,K)$ is then tautological (but see Remarks~\ref{remark:Torbalanced} and \ref{remark:Torbimodule} for further discussion).\end{proof}

\begin{numberedremark}
\label{remark:Torbalanced}
The notation $\Tor^{\FI}_*(V,W)$ requires some justification, since this could denote the left-derived functors of $V\otimes_{\FI}-$ or of $-\otimes_{\FI}W$. Fortunately, the tensor product functor \[{-\otimes_{\FI}-}\colon \FIMod\times \FIopMod\to \ZMod\] is a left-balanced functor in the sense of \cite[Definition~2.7.7]{Weibel}, so by \cite[Exercise~2.7.4]{Weibel} its left-derived functors in the first variable and in the second variable coincide. In other words, these derived functors $\Tor^{\FI}_*(V,W)$ can be computed either from a resolution $V_\bullet$ of $V$ by projective FI-modules, or from a resolution $W_\bullet$ of $W$ by projective $\FIop$-modules, as we would expect.
\end{numberedremark}

\begin{numberedremark}
\label{remark:Torbimodule}
When $W$ is an FI-bimodule, $V\otimes_{\FI} W$ and thus $\Tor^{\FI}_*(V,W)$ are FI-modules, but there is one important point to make. We can compute the FI-module $\Tor^{\FI}_i(V,W)$ from a projective resolution $W_\bullet\to W$ of FI-bimodules, but in fact something much weaker suffices. We do not need the terms $W_i$ of this resolution to be \emph{projective} FI-bimodules; it suffices that each FI-bimodule $W_i$ be ``$\FIop$-projective'', meaning that for each finite set $T\in \Ob \FI$ the $\FIop$-module $(W_i)_T$ is a projective $\FIop$-module.

This is familiar from the situation of $R$-modules: if $M$ is a right $R$-module and $N$ is an $R\text{-}S$-bimodule, then to compute the $S$-modules $\Tor^R_*(M,N)$ from a resolution $N_\bullet\to N$ by $R\text{-}S$-bimodules, it suffices that each $N_i$ be projective (or even flat) as an $R$-module. The reason is that such an $R\text{-}S$-bimodule is acyclic for the functor $M\otimes_R-\colon R\text{-}S\dMod\to S\dMod$; the situation for FI-modules is the same.

The only projective $\FIop$-modules we will need to consider are the co-representable functors $\Z[\Inj(-,U)]=\Z[\Hom_{\FIop}(U,-)]$ for a fixed finite set $U$ (such co-representable functors are always projective). 
\end{numberedremark}

We may therefore describe $H^{\FI}_i(V)$ in a uniform way that applies to all FI-modules $V$ by finding an appropriate resolution $C_\bullet\to K$ of $\FIop$-projective FI-bimodules.

\para{A uniform construction of FI-complexes}
We will make use of the same construction in multiple places below, so we begin by describing this construction in a general context; we are grateful to the referee for suggesting this.

\begin{definition}
\label{def:twisted}
We denote by $\FIarrow$ the twisted arrow category whose objects are pairs $(T,U)$ where $T$ is a finite set and $U\subset T$ is a subset, and where a morphism from $(T,U)$ to $(T',U')$ is an injection $f\colon T\into T'$ such that $f(U)\supseteq U'$.
\end{definition}
Given an $\FIarrow$-module $F$, we will construct two chain complexes of FI-modules. In fact, for any functor $F$ from $\FIarrow$ to any abelian category $\AA$, we construct two chain complexes $\ch^F_\bullet$ and $\chT^F_\bullet$  taking values in $[\FI,\AA]$.

\begin{construction}[The complexes $\ch^F_\bullet$ and $\chT^F_\bullet$]
\label{const:chF}
Given a functor $F\colon \FIarrow\to\AA$, for each $k\geq 0$ define $\chT^F_k\colon \FI\to \AA$ by
\[\chT^F_k(T)=\bigoplus_{f: [k]\into T}F(T,\im f).\]
An FI-morphism $g\colon T\into T'$ defines for each $f\colon [k]\into T$ an $\FIarrow$-morphism $g\colon (T,\im f)\to (T',\im g\circ f)$, and $g_*\colon \chT^F_k(T)\to \chT^F_k(T')$ is given by the induced maps. 

Next, we define the boundary map $\partial\colon \chT^F_k\to \chT^F_{k-1}$. For $k\geq 1$ and $1\leq i\leq k$, let $\delta_i\colon [k-1]\into [k]$ be the ordered injection whose image does not contain $i$. For any $f\colon [k]\into T$, the identity $\id_T$ defines an $\FIarrow$-morphism from $(T,\im f)$ to $(T,\im f\circ \delta_i)$. Let $d_i\colon \chT^F_k\to \chT^F_{k-1}$ be the map induced on each factor by $\id_T\colon (T,\im f)\to(T,\im f\circ \delta_i)$; note that this commutes with the $\FI$-action $g_*$ defined above.

We define $\partial\colon \chT^F_k\to \chT^F_{k-1}$ by $\partial\coloneq \sum (-1)^i d_i$. The familiar formula $\delta_i\circ \delta_j=\delta_{j+1}\circ \delta_i$ for $i\leq j$ implies that  $d_j\circ d_i=d_i\circ d_{j+1}$ by the functoriality of $F$, so $\partial^2 =0$. Therefore the differential $\partial$ makes $\chT^F_\bullet$ a chain complex with values in $[\FI,\AA]$.

We  define the complex $\ch^F_\bullet$ as the quotient of $\chT^F_\bullet$ by the following relations. 
The permutations $\sigma\in S_k$ act on $\chT^F_k$ by precomposition, and breaking up into orbits we have \[\chT^F_k(T)=\bigoplus_{\substack{U\subset T\\\abs{U}=k}}\bigoplus_{f: [k]\iso U}F(T,U).\]
We define $\ch^F_k$ to be the quotient of $\chT^F_k$ by the relations $\sigma_*=(-1)^\sigma$ for all $\sigma\in S_k$; in other words, we pass to the quotient where $S_k$ acts by the sign representation. The functoriality of $F$ guarantees that $\ch^F_k$ is still a functor $\FI\to \AA$.

The individual homomorphisms $d_i$ do not respect these relations, so they do not descend to $\ch^F_k$. However, the alternating sum $\partial=\sum (-1)^i d_i$ does descend to a differential $\partial\colon \ch^F_k\to \ch^F_{k-1}$, and so we obtain a chain complex $\ch^F_\bullet=\cdots\to \ch^F_k\to \ch^F_{k-1}\to \cdots\to \ch^F_0$ with values in $[\FI,\AA]$. Note that on objects we have \[(\ch^F_k)_T=\bigoplus_{\substack{U\subset T\\\abs{U}=k}}F(T,U),\]
where FI-morphisms act with a factor of $\pm 1$ coming from the orientation of the subset $U$.
\end{construction}

\begin{numberedremark}
\label{remark:summand}
When the finite set $T$ is fixed, the following standard argument shows that the chain complex $\ch^F_\bullet(T)$ is a summand of $\chT^F_\bullet(T)$.
Choosing an ordering of $T$, let $\chT_\bullet^{\text{ord}}(T)$ be the subcomplex of $\chT^F_\bullet(T)$ spanned by those summands where $f\colon [k]\into T$ is order-preserving. The differential $\partial$ preserves this subcomplex, and the projection $\chT^F_\bullet(T)\to \ch^F_\bullet(T)$ restricts to an isomorphism $\chT_\bullet^{\text{ord}}(T)\iso \ch^F_\bullet(T)$. However, we emphasize that $\ch^F_\bullet$ is \emph{not} a summand of $\chT^F_\bullet$ when these are considered as complexes of FI-modules. 
\end{numberedremark}

We now use this construction to define a complex $C_\bullet\to K$ of FI-bimodules, which will give us our resolution of $K$.

\begin{definition}[name={The complex $C_\bullet$}]
\label{def:Cbullet}
Given $U\subset T$, let $F(T,U)$ be the $\FIop$-module defined by $F(T,U)_S=\Z[f\colon S\into T\setminus U]$. One easily checks that this defines a functor $F\colon \FIarrow\to \FIopMod$, so \autoref{const:chF} defines a chain complex $C_\bullet\coloneq \ch^F_\bullet$ with values in $[\FI,\FIopMod]$, i.e.\ a complex of FI-bimodules. Concretely, $C_k(S,T)$ is the free abelian group on pairs $(U\subset T,f\colon S\into T)$ where $\abs{U}=k$ and $\im f$ is disjoint from $U$.
\end{definition}

\begin{unnumberedremark}
See \cite[Eq (10)]{CEFN} and the surrounding section for more discussion of this complex. A caution: we could similarly have defined a complex $\chT^F_\bullet$ of FI-bimodules, but be warned that the $\FI^{\op}\times \FI$-module $B_V$ discussed following \cite[Corollary 2.18]{CEFN} is \emph{not} isomorphic to $\chT^F_\bullet$, although they contain much the same information.
\end{unnumberedremark}

\para{The resolution $C_\bullet\to K$}
We consider the augmentation map $\partial\colon C_0\to K$ defined by
\[C_0(S,T)\ni (\emptyset, f\colon S\into T)\mapsto \begin{cases}f\in \Bij(S,T)&\text{if }\abs{S}=\abs{T}\\
0&\text{if }\abs{S}<\abs{T}\end{cases}\in K(S,T).\]
Since $C_1(S,T)$ has basis $(\{u\}\subset T, f\colon S\into T\setminus\{u\})$, the composition $\partial^2\colon C_1\to C_0\to K$ is 0.
Therefore this augmentation extends $C_\bullet$ to a complex \[\cdots\to C_1\to C_0\to K\to 0.\]

\begin{proposition}
\label{prop:resolution}
The complex $C_\bullet\to K$ is a resolution of $K$ by FI$^{\op}$-projective FI-bimodules. As a consequence, given any FI-module $V$, the FI-homology of $V$ is computed by the FI-chain complex $V\otimes_{\FI} C_\bullet$: \[H_i^{\FI}(V)=H_i(V\otimes_{\FI} C_\bullet)\]
\end{proposition}
\begin{proof}
We first verify that $C_\bullet\to K$ is a resolution, i.e.\ that $H_0(C_\bullet)\iso K$ and $H_*(C_\bullet) = 0$ for $*>0$. It suffices to check this pointwise, so fix finite sets $S$ and $T$ and consider the chain complex of abelian groups $C_\bullet(S,T)$.

For each $h\colon S\into T$, let $C^h_k(S,T)$ be the summand of $C_k(S,T)$ spanned by the elements of the form $(U,h)$. The differential $\partial$ preserves this summand, so we have a direct sum decomposition $C_\bullet(S,T)=\bigoplus_{h:S\into T} C^h_\bullet(S,T)$. Similarly, let $K^h(S,T)$ be the corresponding summand of $K(S,T)$; concretely, this summand is isomorphic to $\Z$ if $h$ is bijective and $0$ otherwise. It therefore suffices to show for fixed $h\colon S\into T$ that $C^h_\bullet(S,T)$ is a resolution of $K^h(S,T)$.

Let $\Delta^{T-h(S)}$ be the $(\abs{T-h(S)}-1)$-dimensional simplex with vertex set $T-h(S)$, and let $\widetilde{C}_\bullet(\Delta^{T-h(S)})$ be its reduced cellular chain complex. A basis for $C^h_k(S,T)$ is given by the $k$-element subsets $U$ of $T-h(S)$, oriented appropriately. In other words, we can identify $C^h_k(S,T)\iso \widetilde{C}_{k-1}(\Delta^{T-h(S)})$, and this extends to an isomorphism of chain complexes $C^h_\bullet(S,T)\iso \widetilde{C}_{\bullet-1}(\Delta^{T-h(S)})$.

As long as $T-h(S)$ is nonempty, the simplex $\Delta^{T-h(S)}$ is contractible, so $H_*(C^h_\bullet(S,T))\iso\widetilde{H}_{*-1}(\Delta^{T-h(S)})=0$ for all $*\geq 0$. Since $K^h(S,T)=0$ when $h$ is not bijective, this is as desired. In the remaining case when $h$ is a bijection and $\Delta^{T-h(S)}$ is empty, the only nonzero term of this resolution is $C^h_0(S,T)\iso \widetilde{C}_{-1}(\emptyset)\iso \Z\iso K^h(S,T)$, which again is as desired.

We next verify that the FI-bimodules $C_k$ are $\FIop$-projective, meaning that for each finite set $T$ the $\FIop$-module $C_k(-,T)$ is a projective $\FIop$-module. For a fixed $k$-element subset $U\subset T$, let $C^U_k(S,T)$ be the summand of $C_k(S,T)$ spanned by elements $(U,f\colon S\into T\setminus U)$. These summands are preserved by $\FIop$-morphisms, so this defines a summand $C^U_k(-,T)$ of the $\FIop$-module $C_k(-,T)$. This summand $C^U_k(-,T)$ is isomorphic to the co-representable functor $\Z[\Inj(-,T\setminus U)]=\Z[\Hom_{\FIop}(T\setminus U,-)]$. 
Since $C_k(-,T)=\bigoplus C^U_k(-,T)$, this shows that $C_k(-,T)$ is a projective $\FIop$-module, as desired.

It now follows from \autoref{prop:TorK} and \autoref{remark:Torbimodule} that $H_i^{\FI}(V)=H_i(V\otimes_{\FI} C_\bullet)$.
\end{proof}
\begin{numberedremark}
A result essentially equivalent to the conclusion of \autoref{prop:resolution} has been proved independently in a recent preprint of Gan--Li~\cite[Th 1]{GanLi}.
\end{numberedremark}

\begin{numberedremark}
It is possible to interpret $C_\bullet$ as the ``Koszul resolution of $\FI$ over $K$'', thinking of $f\in \Hom_{\FI}(S,T)$ as graded by $\abs{T}-\abs{S}=\abs{T-f(S)}$. Moreover, under Schur-Weyl duality $C_\bullet$ corresponds to the classical Koszul resolution of $\Sym^*V$ by $\bwedge^*V^\vee\otimes \Sym^*V$. For reasons of space we will not pursue this further here; see \cite[\S 6]{SamSnowdenGL} for more details, including strong theorems regarding this Koszul duality for FI-modules over $\mathbb{C}$.
\end{numberedremark}

We can now prove  \autoref{main:colimit}.
\begin{proof}[Proof of \autoref{main:colimit}]
The desired result states for a particular integer $N$ (namely the maximum of $\deg H^{\FI}_0(V)$ and $\deg H^{\FI}_1(V)$) that \begin{equation}
\label{eq:thmB}
\colim_{\substack{S\subset T\\\abs{S}\leq N}} V_S=V_T\qquad\text{for all finite sets }T.
\end{equation}
We introduced in \cite[Definition~2.19]{CEFN} a certain complex of FI-modules $\widetilde{S}_{-\bullet}V$, and 
combining our earlier results \cite[Theorem~C and Corollary 2.24]{CEFN} shows that \eqref{eq:thmB} holds if and only if $H_0(\widetilde{S}_{-\bullet}V)_n=0$ and $H_1(\widetilde{S}_{-\bullet}V)_n=0$ for all $n>N$.

Our main goal will be therefore to prove that $V\otimes_{\FI} C_\bullet\iso \widetilde{S}_{-\bullet}(V)$. Given this, we know that \[H_i(\widetilde{S}_{-\bullet}V)\iso H_i(V\otimes_{\FI} C_\bullet)\iso \Tor^{\FI}_i(V,K)\iso H^{\FI}_i(V),\]
where the second isomorphism holds by \autoref{prop:resolution} and the third isomorphism holds by \autoref{prop:TorK}. Therefore \eqref{eq:thmB} holds if and only if $H^{\FI}_0(V)_n=0$ and $H_1^{\FI}(V)_n=0$ for all $n>N$. In other words, the desired condition \eqref{eq:thmB} holds exactly when $\deg H^{\FI}_0(V)\leq K$ and $\deg H^{\FI}_1(V)\leq N$, which is precisely what the theorem claims.

Recall from \autoref{def:twisted} the category $\FIarrow$.
For any FI-module $V$, we can define an $\FIarrow$-module $F_V$ by $F_V(T,U)=V_{T\setminus U}$, since an $\FIarrow$-morphism $(T,U)\to (T',U')$ restricts to an inclusion $T\setminus U\into T'\setminus U'$. We first show that the complex of FI-modules $V\otimes_{\FI} C_\bullet$ coincides with the complex $\ch^{F_V}_\bullet$ of \autoref{const:chF}.

We saw in the proof of \autoref{prop:resolution} that $C_k(-,T)=\bigoplus_{\abs{U}=k} C_k^U(-,T)$ where $C_k^U(-,T)$ is the co-representable functor $\Z[\Hom_{\FIop}(T\setminus U,-)]$. By the Yoneda lemma, the tensor of $V$ with a functor co-represented by $R$ is simply $V_R$. Therefore as abelian groups we have an isomorphism
\[(V\otimes_{\FI} C_k)_T\iso \bigoplus_{\abs{U}=k}V_{T\setminus U}\iso \ch^{F_V}_k(T)\]
Checking the morphisms and differential, we see that $V\otimes_{\FI} C_\bullet$ and $\ch^{F_V}_\bullet$ coincide as chain complexes of FI-modules.

We conclude by showing that $\ch^{F_V}_\bullet$ coincides with $\widetilde{S}_{-\bullet}(V)$. We will in fact show that $\chT^{F_V}_\bullet$ coincides with the $S_n$-complex of FI-modules $B_\bullet(V)$ of 
\cite[Eq.\ (10)]{CEFN}.
As an abelian group \[\chT^{F_V}_k(T)= \bigoplus_{f: [k]\into T}V_{T\setminus \im f},\] 
and $B_k(V)_T$ is defined by the same formula \cite[Definition~2.9]{CEFN}.
Given an injection $g\colon T\into T'$, unwinding \autoref{const:chF} shows that the  map $g_*\colon \chT^{F_V}_k(T)\to \chT^{F_V}_k(T')$ sends the summand labeled by $f$ to the summand labeled by $g\circ f\colon [k]\into T'$ by the map $(g|_{T\setminus \im f})_*\colon V_{T\setminus \im f}\to V_{T'\setminus \im g\circ f}$. This is precisely the FI-structure on $B_k(V)$.
Finally, the maps $d_i$ of \autoref{const:chF} agree with those defined just before \cite[Eq.\ (10)]{CEFN}, so the resulting differentials $\partial=\sum(-1)^i d_i$ agree as well.

The $S_k$-actions on $\chT^{F_V}_k$ and on $B_k(V)$ agree, and $\ch^{F_V}_k$ and $\widetilde{S}_{-k}(V)$ are respectively obtained from these by tensoring over $S_k$ with the sign representation. So we conclude that $V\otimes_{\FI} \Ctilde_\bullet\iso \ch^{F_V}_\bullet$ is isomorphic to $ \widetilde{S}_{-\bullet}(V)$ as chain complexes of FI-modules, as desired.\end{proof}

\subsection{Homology of congruence subgroups}
\label{sec:congruence}
In this section, we state and prove \autoref{thm:congruencestronger}, a more general version of \autoref{main:congruence} from the introduction.

Let $R$ be a commutative ring satisfying Bass's stable range condition $SR_{d+2}$, and fix a proper ideal $\p\subsetneq R$. (We use Bass's indexing convention, under which a field satisfies $SR_2$, and any Noetherian $d$-dimensional ring satisfies $SR_{d+2}$.) Let $\Gamma_n(\p)$ be the congruence subgroup defined by the exact sequence of groups
\[1\to \Gamma_n(\p)\to \GL_n(R)\to \GL_n(R/\p)\]
As explained in \cite[\S3]{CEFN}, these groups form an FI-group $\Gamma(\p)$ (a functor $\FI\to \Groups$ satisfying $\Gamma(\p)_T\iso \Gamma_{\abs{T}}(\p)$), and thus their integral homology forms an FI-module
\[\HH_k\coloneq H_k(\Gamma(\p);\Z).\]
\begin{theorem-prime}{main:congruence}
\label{thm:congruencestronger}
Let $R$ be a commutative ring satisfying Bass's stable range condition $SR_{d+2}$, and let $\p\subsetneq R$ be a proper ideal. Then for all $k\geq 2$,
\[\deg H_0^{\FI}(\HH_k) \leq 2^{k-2}(2d+9)-2\qquad\text{and}\qquad \deg H_1^{\FI}(\HH_k) \leq 2^{k-2}(2d+9)-1.\]
In particular, for all $n\geq 0$ and all $k\geq 0$ we have
\begin{equation}
\label{eq:Putmangen}
H_k(\Gamma_n(\p);\Z)=\colim_{\substack{S\subset [n]\\\abs{S}<2^{k-2}(2d+9)}} H_k(\Gamma_S(\p);\Z).
\end{equation}
\end{theorem-prime}
\autoref{main:congruence} is the special case of \autoref{thm:congruencestronger} when $R=\Z$. Indeed, any Dedekind domain $R$ satisfies Bass's condition $SR_3$ (i.e.\ $SR_{d+2}$ for $d=1$), yielding the bound $\abs{S}<11\cdot 2^{k-2}$ in \autoref{main:congruence}. Note that although we take group homology with integer coefficients in the statement of \autoref{thm:congruencestronger}, these coefficients could be replaced by any other abelian group; the proof applies unchanged.

By the \emph{stable range}, we mean the range $n\geq 2^{k-2}(2d+9)$ where the description \eqref{eq:Putmangen} is not vacuous. 
Our stable range is slightly better than that of \cite{Putman}, where Putman obtained the range $n\geq 2^{k-2}(2d+16)-3$. For example, \cite[Theorem~B]{Putman} gives  for a Dedekind domain $R$ the stable range $n\geq 18\cdot 2^{k-2}-3$, while \autoref{thm:congruencestronger} gives the stable range $n\geq 11\cdot 2^{k-2}$. 

\begin{proof}[Proof of \autoref{thm:congruencestronger}]
To avoid confusion with the homology of a chain complex, in this section we write $H^{\FI}_p(W)$ for the FI-homology of an FI-module $W$ (which in previous sections was denoted simply $H_p(W)$).

An action of an FI-group $\Gamma$ on an FI-module $M$ is a collection of actions of $\Gamma_T$ on $M_T$ that are consistent with the FI-structure. 
Given such an action, the coinvariants form an FI-module $\Z\otimes_{\Gamma}M$, whose components are simply $\Z\otimes_{\Gamma_T}M_T$. The left-derived functors $H_i(\Gamma;M)$ are simply the FI-modules defined by $H_i(\Gamma;M)_T\coloneq H_i(\Gamma_T;M_T)$.
In the special case when $M=M(0)$ and the action is trivial, we write $\HH_i(\Gamma)$; this is the group homology, considered as an FI-module $\HH_i(\Gamma)_T\coloneq H_i(\Gamma_T;\Z)$.

We will need the following proposition, which constructs for any FI-group a spectral sequence based on the FI-homology of its group homology.
\begin{proposition}
\label{pr:Gammacomplex}
To any FI-group $\Gamma$ there is naturally associated an explicit FI-chain complex $X^\Gamma_\bullet$ on which $\Gamma$ acts, for which we have a spectral sequence
\[E^2_{pq}=H_p^{\FI}(\HH_q(\Gamma))\implies H_{p+q}(\Gamma;X^\Gamma_\bullet).\]
\end{proposition}
\begin{proof}
Recall from \autoref{def:twisted} the category $\FIarrow$ used in \autoref{const:chF}. Define the $\FIarrow$-module $A$ by $A(T,U)=\Z[\Gamma_T/\Gamma_{T\setminus U}]$. An $\FIarrow$-morphism $f\colon (T,U)\to (T',U')$ has $f(U)\supseteq U'$, so the induced map $f_*\colon \Gamma_T\to \Gamma_{T'}$ satisfies $f_*(\Gamma_{T\setminus U})\subset \Gamma_{T'\setminus U'}$, verifying that $A$ is indeed an $\FIarrow$-module.

The FI-chain complex $X_\bullet=X^\Gamma_\bullet$ we are interested in will be the FI-chain complex $X_\bullet\coloneq \ch^A_\bullet$ arising from $A$ via \autoref{const:chF}:
\[
X_k(T)=\bigoplus_{\substack{U\subset T\\\abs{U}=k}}\Z[\Gamma_T/\Gamma_{T\setminus U}]\]
For each $T$ the obvious action of $\Gamma_T$ on $\Z[\Gamma_T/\Gamma_{T\setminus U}]$ induces an action of $\Gamma_T$ on $X_k(T)$. The FI-module structure on $X_k$ is induced by the FI-structure maps $\Gamma_T\to \Gamma_{T'}$, and the differential $\partial$ descends from the \emph{identity} on $\Gamma_T$. Therefore the action of $\Gamma_T$ on $X_k(T)$ is  compatible with both, giving an action of the FI-group $\Gamma$ on the FI-chain complex $X_\bullet$.

From this action we obtain two spectral sequences converging to the homology
$H_*(\Gamma;X_\bullet)$ of the complex $X_\bullet$:
\begin{align*}
\overline{E}^2_{pq}&=H_p(\Gamma;H_q(X_\bullet))\implies H_{p+q}(\Gamma;X_\bullet)\\
E^1_{pq}&=H_q(\Gamma;X_p\qquad) \implies H_{p+q}(\Gamma;X_\bullet)
\end{align*}
The desired spectral sequence mentioned in the proposition is the second one (though we will use the first spectral sequence later). It remains to identify $E^2_{pq}$ with $H_p^{\FI}(\HH_q(\Gamma))$, so let us compute $E^1_{pq}=H_q(\Gamma;X_p)$.

By definition $X_p(T)$ is a direct sum of factors $\Z[\Gamma_T/\Gamma_{T\setminus U}]$. By Shapiro's lemma, the contribution of such a factor to $H_q(\Gamma_T; X_p(T))$ is precisely $H_q(\Gamma_{T\setminus U})= \HH_q(\Gamma)_{T\setminus U}$. 
We find that
\[H_q(\Gamma;X_p)_T=H_q(\Gamma_T; X_p(T))=\bigoplus_{\substack{U\subset T\\\abs{U}=p}}\HH_q(\Gamma)_{T\setminus U}=(\HH_q(\Gamma)\otimes_{\FI} C_p)_T,\] where the last equality comes from the proof of \autoref{main:colimit}.
We conclude that
\[E^1_{pq}=H_q(\Gamma;X_p)\iso \HH_q(\Gamma)\otimes_{\FI} C_p.\]
Moreover, the differential $d^1\colon H_q(\Gamma;X_p)\to H_q(\Gamma;X_{p-1})$ is induced by $\partial\colon X_p\to X_{p-1}$, and comparing the definitions of $X_\bullet$ and $C_\bullet$ shows that indeed $(E^1_{pq},d^1)=(\HH_q(\Gamma)\otimes_{\FI} C_\bullet,\partial)$. By \autoref{prop:resolution} we conclude that, as claimed,
\[E^2_{pq}=\Tor_p^{\FI}(\HH_q(\Gamma),K)=H^{\FI}_p(\HH_q(\Gamma)) \implies H_{p+q}(\Gamma;X_\bullet).\qedhere\]
\end{proof}

We now continue with the proof of \autoref{thm:congruencestronger}. Returning to the notation of that theorem, let $\Gamma$ be the congruence FI-group $\Gamma(\p)$, and $\HH_k=\HH_k(\Gamma(\p))$ its group homology. We would like to apply \autoref{pr:Gammacomplex}, but to do this we need to bound the equivariant homology $H_{p+q}(\Gamma;X_\bullet)$.
We can do this using the other spectral sequence $\overline{E}^2_{pq}=H_p(\Gamma;H_q(X_\bullet))\implies H_{p+q}(\Gamma;X_\bullet)$ if we can bound $H_q(X_\bullet)$. And fortunately, this complex $X_\bullet$ (or a complex quite close to  it) has already been considered by Charney!

In \autoref{pr:Gammacomplex} we defined $X_\bullet=\ch^A_\bullet$ based on the functor $A(T,U)=\Z[\Gamma_T/\Gamma_{T\setminus U}]$. Let $\XX_\bullet\coloneq \chT^A_\bullet$ be the ordered version of this complex; concretely, we can write
\[\XX_k(T)=\bigoplus_{(t_1,\ldots,t_k)\subset T}\!\!\!\Z[\Gamma_T/\Gamma_{T-\{t_1,\ldots,t_k\}}].\]
In the foundational paper \cite{Charney}, Charney considered (in the case $T=\{1,\ldots,n\}$) a complex $Y_\bullet(T)$ that is  similar to $\XX_\bullet(T)$ but somewhat larger. Her key technical result in that paper is that $Y_\bullet(T)$ is $q$-acyclic if $\abs{T}\geq 2q+d+1$. Moreover \cite[Proposition~3.2]{Charney} implies that $Y_q(T)$ coincides with $\XX_q(T)$ as long as $\abs{T}\geq q+d$, so Charney's result\footnote{Note that our indexing differs from Charney's in that her complex has $\XX_q(T)$ in degree $q-1$; this is why we have $2q+d+1$ and $q+d$ in place of her $2q+d+3$ and $q+d+1$, respectively.} implies that $\XX_\bullet(T)$ is $q$-acyclic if $\abs{T}\geq 2q+d+1$. By \autoref{remark:summand} we know that $X_\bullet(T)$ is a summand of $\XX_\bullet(T)$, so $X_\bullet(T)$ is $q$-acyclic in the same range.
Said differently, $H_q(X_\bullet)_T= 0$ for $\abs{T}>2q+d$; that is, $\deg H_q(X_\bullet)\leq 2q+d$. 

Any FI-module $M$ with $\deg M\leq N$ automatically has $\deg H_i(\Gamma;M)\leq N$ for all $i$, since $H_*(\Gamma;M)=H_*(\Gamma_T;0)=0$ when $\abs{T}>N$. Therefore Charney's bound  $\deg H_q(X_\bullet)\leq 2q+d$ implies \[\deg \overline{E}^2_{pq}=\deg H_p(\Gamma;H_q(X_\bullet))\leq 2q+d\] for all $p$.
Since this spectral sequence converges to $ \overline{E}^2_{pq}\implies H_{p+q}(\Gamma;X_\bullet)$, we conclude that \begin{equation}
\label{eq:Hkbound}
\deg H_k(\Gamma;X_\bullet)\leq 2k+d.
\end{equation}

The bound \eqref{eq:Hkbound} marks the end of the input from topology in this proof. The remainder of the proof is just careful bookkeeping and repeatedly applying \autoref{main:regularity} to our spectral sequence of FI-modules \[E^2_{pq}=H_p^{\FI}(\HH_q)\implies H_{p+q}(\Gamma;X_\bullet).\] In fact, this bookkeeping can be formulated as the following completely general statement: 

\para{Claim} Consider  a spectral sequence of FI-modules $E^2_{pq}\implies V_{p+q}$ converging to FI-modules $V_k$ satisfying $\deg V_k\leq 2k+d$ for some integer $d\geq 0$. Suppose that for all $q$ we know that
\begin{equation}
\tag{$\alpha$}
\label{eq:mainthmsub}
\deg E^2_{pq}\leq \deg E^2_{0q}+\deg E^2_{1q}-1+p,
\end{equation}
and suppose for simplicity that $E^2_{p0}=0$ for $p>0$.
Then for all $k\geq 2$ we have
\begin{equation}
\label{eq:H01bounds}
\deg E^2_{0k}\leq 2^{k-2}(2d+9)-2\text{\quad and \quad}\deg E^2_{1k}\leq 2^{k-2}(2d+9)-1.
\end{equation}

We would like to prove this claim \eqref{eq:H01bounds}
 by induction on $k$ for all $k\geq 2$, but we need to modify it slightly so it holds in the base cases $k\in \{0,1\}$ as well. Therefore we will prove along the way that
\begin{equation}
\label{eq:Hpbounds} 
\forall p\geq 2:\qquad\deg E^2_{pk}\leq 2^{k-1}(2d+9)-4+p
\end{equation}
holds for all $k \geq 0$. Notice that \eqref{eq:H01bounds} + \eqref{eq:mainthmsub} $\implies$ \eqref{eq:Hpbounds}, so this only requires additional work in the base cases when $k \in \set{0,1}$. We first prove  \eqref{eq:Hpbounds} in these base cases, and then prove by induction on $k$ that both \eqref{eq:H01bounds} and \eqref{eq:Hpbounds} hold for all $k \geq 2$.

\para{Case $k=0$} Our assumption that $E^2_{p0}=0$ for all $p\geq 1$ implies $\deg E^2_{p0}=-\infty$, so   \eqref{eq:Hpbounds} holds.

\para{Case $k=1$} Since $E^2_{3,0}=0$ and $E^2_{4,0}=0$, the spectral sequence degenerates at $E^2$ for $E^2_{0,1}$ and $E^2_{1,1}$, yielding $E^2_{0,1}=E^\infty_{0,1}=V^1$ and $E^2_{1,1}=E^\infty_{1,1}\subset V^2$. Since $\deg V^1\leq 2+d$ and $\deg V^2\leq 4+d$, we conclude that $\deg E^2_{0,1}\leq d+2$ and $\deg E^2_{1,1}\leq d+4$. Applying the assumption \eqref{eq:mainthmsub}, we conclude that $\deg E^2_{p,1}\leq 2d+5+p$ for all $p\geq 2$; this is precisely the bound \eqref{eq:Hpbounds} in the case $k=1$.

\para{General case} Let $N_{p,m}\coloneq 2^{m-1}(2d+9)-4+p$ be the bound occurring in \eqref{eq:Hpbounds}. Fix $k\geq 2$, and assume by induction that \eqref{eq:Hpbounds} holds for all $m < k$; that is, $\deg E^2_{p,m}\leq N_{p,m}$ for all $p\geq 2$ and all $m<k$.

Now consider the entry $E^2_{0,k}$. Since $E^\infty_{0,k}$ is a constituent of $V^k$, we have $\deg E^\infty_{0,k}\leq \deg V^k\leq 2d+k$. No nontrivial differential has source $E^r_{0,k}$, but we have differentials $d^r\colon E^r_{r,k-r+1}\to E^r_{0,k}$. The  maximum of $N_{r,k-r+1}$ over $r\geq 2$ occurs when $r=2$, when we have $N_{2,k-1}=2^{k-2}(2d+9)-2$. Therefore for all $r\geq 2$ the sources of these differentials satisfy $\deg E^r_{r,k-r+1}\leq 2^{k-2}(2d+9)-2$. Since $\deg E^\infty_{0,k}\leq 2d+k< 2^{k-2}(2d+9)-2$, we conclude that $\deg E^2_{0,k}\leq 2^{k-2}(2d+9)-2$, as claimed in \eqref{eq:H01bounds}.
Similarly, the degrees of the sources of the differentials $d^r\colon E^r_{1+r,k-r+1}\to E^r_{1,k}$ are bounded above by $N_{3,k-1}=2^{k-2}(2d+9)-1$. Since $\deg E^\infty_{1,k}\leq \deg V^{k+1}\leq 2d+k+1< 2^{k-2}(2d+9)-1$, we conclude that $\deg E^2_{1,k}\leq 2^{k-2}(2d+9)-1$, as claimed in \eqref{eq:H01bounds}. 

Now applying the assumption \eqref{eq:mainthmsub} to \eqref{eq:H01bounds}, we conclude that \eqref{eq:Hpbounds} holds for $k$ as well. This concludes the proof of the claim.

\medskip
We now finish the proof of \autoref{thm:congruencestronger} by applying this claim to the spectral sequence $E^2_{pq}=H_p^{\FI}(\HH_q)\implies H_{p+q}(\Gamma;X_\bullet)$ of \autoref{pr:Gammacomplex}. The hypothesis $E^2_{p0}=0$ of the claim is satisfied because $\HH_0$ is the free FI-module $\HH_0\iso M(0)$, so $E^2_{p0}=H_p^{\FI}(\HH_0)=0$ for $p>0$. The assumption \eqref{eq:mainthmsub} is precisely the statement of \autoref{main:regularity}, and the bound $\deg H_k(\Gamma;X_\bullet)\leq 2k+d$ was obtained in \eqref{eq:Hkbound} above.
 
The description \eqref{eq:Putmangen} for $k\geq 2$ follows from \eqref{eq:H01bounds} by \autoref{main:colimit}. The only thing that remains is some arithmetic to check that \eqref{eq:Putmangen}  holds for $k=0$ and $k=1$ as well.

For $k=0$ this is trivial, since $\HH_0=M(0)$ is free: this means  $\deg H_0^{\FI}(\HH_0)=0$ and $\deg H_1^{\FI}(\HH_0)=-\infty$, so  \autoref{main:colimit} then gives an identification as in \eqref{eq:Putmangen} over $\abs{S}\leq 0$. Since $d\geq 0$, we have $2^{0-2}(2d+9)\geq \frac{9}{4}>1$, so the bound in \eqref{eq:Putmangen} holds.

Similarly, for $k=1$ we saw in the proof above that $\deg H_0^{\FI}(\HH_1)=\deg E^2_{0,1}\leq 2+d$ and $\deg H_1^{\FI}(\HH_1)=\deg E^2_{1,1}\leq 4+d$, so \autoref{main:colimit} gives an identification as in \eqref{eq:Putmangen} over $\abs{S}\leq 4+d$. For integer $m$ the conditions $m<2^{1-2}(2d+9)=d+\frac{9}{2}$ and $m\leq d+4$ are equivalent, so again the bound in \eqref{eq:Putmangen} follows.
\end{proof}

We close with a variant of \autoref{thm:congruencestronger} which has been used by Calegari--Emerton \cite{CalegariEmerton} in their study of completed homology. An inclusion of ideals $\q\subset \p$ induces an inclusion $\Gamma_n(\q)\subset \Gamma_n(\p)$, so given an inverse system of ideals such as $\cdots\subset \p^i\subset \cdots\subset \p^2\subset \p$ we can consider the inverse limit ${\displaystyle{\lim_{\leftarrow} H_k(\Gamma_n(\p^i))}}$ of the homology of the corresponding congruence subgroups.
\begin{theorem-doubleprime}{main:congruence}
\label{thm:congruencecompleted}
Let $R$ be the ring of integers in a number field, and let $(\p_i)_{i\in I}$ be an inverse system of proper ideals in $R$. Fix $N > 1$. Then for all $n\geq 0$ and all $k\geq 0$ we have
\[
\lim_{\substack{\leftarrow\\i\in I}}H_k(\Gamma_n(\p_i);\Z/N)=\colim_{\substack{S\subset [n]\\\abs{S}<11\cdot 2^{k-2}}}\ \ \lim_{\substack{\leftarrow\\i\in I}}\ H_k(\Gamma_S(\p_i);\Z/N).
\]
\end{theorem-doubleprime}
\begin{proof}
Any number ring $R$ is a Dedekind domain, so $R$ satisfies Bass's stable range condition $SR_3$. Therefore for any $n\geq 0$ and any $k\geq 0$, we can deduce from \autoref{thm:congruencestronger} that
\[\lim_{\substack{\leftarrow\\i\in I}}H_k(\Gamma_n(\p_i);\Z/N)= \lim_{\substack{\leftarrow\\i\in I}}\ \ \colim_{\substack{S\subset [n]\\\abs{S}<11\cdot 2^{k-2}}}\ H_k(\Gamma_S(\p_i);\Z/N).\] It remains to check that we can exchange the limit and colimit. This is of course not true in general, but we can verify that it is true in this case as follows. The existence of the Borel--Serre compactification \cite{BorelSerre} implies that $H_k(\Gamma_n(\p);\Z/N)$ is a finitely-generated $\Z/N$-module for any $\p\subset R$. This is enough to give the desired result: since this colimit is over a finite poset, it can therefore be written as a coequalizer of finitely generated $\Z/N$-modules. The limit of the coequalizers is the coequalizer of the limits (any inverse system of finite abelian groups satisfies the Mittag--Leffler condition, so the $\lim^1$ term vanishes), which is to say that the limit and colimit can be exchanged as desired.
\end{proof}

\begin{footnotesize}
\noindent
\begin{tabular*}{\linewidth}[t]{@{}p{\widthof{Department of Mathematics}+1.5in}@{}p{\widthof{Department of Mathematics}+0.5in}}
{\raggedright
Thomas Church\\
Department of Mathematics\\
Stanford University\\
450 Serra Mall\\
Stanford, CA 94305\\
\myemail{church@math.stanford.edu}}
&
{\raggedright
Jordan S. Ellenberg\\
Department of Mathematics\\
University of Wisconsin\\
480 Lincoln Drive\\
Madison, WI 53706\\
\myemail{ellenber@math.wisc.edu}}
\end{tabular*}\hfill
\end{footnotesize}

\end{document}